\documentclass[final, nomarks, pdftex ]{dmtcs-episciences}

\usepackage[utf8]{inputenc}
\usepackage{subfigure}
\usepackage[round]{natbib}

\usepackage{amsmath}
\usepackage[dvipsnames]{xcolor}
\usepackage{url,amssymb,enumerate,tikz,scalefnt,mathrsfs}
\usepackage{color}   
\usepackage{hyperref}
\DeclareFontFamily{OT1}{pzc}{}
\DeclareFontShape{OT1}{pzc}{m}{it}{ <-> s*[1.2] pzcmi7t }{}
\DeclareMathAlphabet{\mathpzc}{OT1}{pzc}{m}{it}

\usetikzlibrary{arrows, arrows.meta, automata, trees}
\usetikzlibrary{shapes,shapes.geometric,fit,calc,positioning,automata}
\usepackage{enumitem}
\usepackage{comment}

\newtheorem{thrm}{Theorem}[section]
\newtheorem{lemma}[thrm]{Lemma}
\newtheorem{prop}[thrm]{Proposition}
\newtheorem{cor}[thrm]{Corollary}

\newtheorem{defn}[thrm]{Definition}
\newtheorem{ex}[thrm]{Example}
\newtheorem{que}[thrm]{Question}

\newtheorem{rem}[thrm]{Remark}

\DeclareMathOperator{\Av}{Av}

\newcommand{\Harpoon}[1]{\hspace{-0.3mm}\! \upharpoonright_{#1}}

\newcommand{\h}{\hspace{0.75mm}}

\newenvironment{thmenumerate}{
\begin{enumerate}[label=\textup{(\roman*)}, widest=(ii), leftmargin=9mm,itemsep=1mm,topsep=0mm]}{
\end{enumerate}}

\newcommand{\N}{\mathbb{N}}

\newcommand{\mc}{\mathcal}
\newcommand{\Rev}[1]{{\color{black}#1}}

\author{V. Ironmonger \and N. Ru\v{s}kuc}
\title[Combinatorial structures under consecutive orders]{Well quasi-order and atomicity for combinatorial structures under consecutive orders}

\affiliation{University of St Andrews, St Andrews, United Kingdom\\}


\begin{document}

\publicationdata{vol. 28:1, Permutation Patterns 2025}{2026}{6}{10.46298/dmtcs.16905}{2025-11-12; 2025-11-12; 2026-04-22; 2026-06-02}{2026-06-19}

\maketitle

\begin{abstract}
\vspace{3mm}
  We consider partially ordered sets of combinatorial structures under consecutive orders, meaning that two structures are related when one embeds in the other such that `consecutive' elements remain consecutive in the image.  Given such a partially ordered set, we may ask decidability questions about its \emph{avoidance sets}: subsets defined by a finite number of forbidden substructures.  Two such questions ask, given a finite set of structures, whether its avoidance set is well quasi-ordered (i.e. contains no infinite antichains) or atomic (i.e. cannot be expressed as the union of two proper subsets).  Extending some recent new approaches, we will establish a general framework, which enables us to answer these problems for a wide class of combinatorial structures, including graphs, digraphs and collections of relations.
\end{abstract}


\section{Introduction}

The purpose of this paper is to investigate some general conditions which render the well quasi-order and atomicity problems decidable for various partially ordered sets (or \emph{posets}) of combinatorial structures under consecutive orders.  Many combinatorial structures can be viewed as relational structures, and it is this viewpoint that we will take.  For example, a digraph consists of a set with a single binary relation, while graphs can be viewed as sets with symmetric binary relations.  Various orderings for such relational structures have been studied, such as embedding orderings (e.g. \cite{Ding}), strong embedding orderings (e.g. \cite{Braun2}), homomorphic image orderings (e.g. \cite{HRhomo}), and homomorphism orderings (e.g. \cite{homo1}).  

We will concentrate on \emph{consecutive orderings}, a variation of (strong) embedding orderings in which embeddings are required to preserve an underlying linear order.  Intuitively, this can be thought of as assigning points of structures the numbers $1, \dots, n$ and requiring embeddings to preserve the natural linear order on $\mathbb{N}$.  Consecutive orders arise naturally for permutations and words, where the underlying linear order gives the `left to right' order of elements, and these posets have been studied in some depth (see \cite{elizalde} for a survey on permutations and \cite{kitaev},  Chapter 5, for information on both).  We will extend the notion of consecutive orderings to combinatorial structures more generally.

The set of all structures of a certain type becomes a poset with the consecutive order, and it is these posets that we will investigate.  We will consider certain distinguished subsets, namely downward closed subsets, also called \emph{avoidance sets}.  Informally, given a finite set $B$, the avoidance set $\Av(B)$ is the subset of structures which do not contain any element of $B$ as a substructure.

We will study two order-theoretic properties for \Rev{avoidance sets of} these posets: well quasi-order, often shortened to \emph{wqo} (which corresponds to the absence of infinite antichains), and atomicity (the property of being indecomposable as the union of two proper, downward closed subsets).  We present our findings within the framework of two decidability questions:

\begin{itemize}
    \item The \emph{well quasi-order problem}: is it decidable, given a finite set $B$, whether $\Av(B)$ is well quasi-ordered? 
    \item The \emph{atomicity problem}: is it decidable, given a finite set $B$, whether $\Av(B)$ is atomic?
\end{itemize}

Two major theorems concern the well quasi-order problem for certain posets of combinatorial structures.  It follows from Higman's celebrated lemma \cite{higman} that the wqo problem is trivially decidable for words under the non-consecutive subword order.  Similarly, the Graph Minor Theorem \cite{graphminor} asserts that the poset of graphs under the minor order is wqo, rendering the wqo problem trivially decidable.  On the other hand, the wqo problem is notably open for graphs under the induced subgraph order and permutations under the classical subpermutation order.  Details on wqo for various combinatorial posets can be found in \cite{HucRus}.

The atomicity problem is decidable for equivalence relations and words under non-consecutive orders (\cite{ir}, \cite{AtomicityAL}), and open for permutations under the non-consecutive subpermutation order.  A recent significant contribution by Braunfeld shows that the atomicity problem is undecidable for 3-dimensional permutations \cite{Braun1}, which may indicate a similar undecidable result for the permutation case.  Braunfeld has also proved undecidability of the atomicity problem for finite graphs under the induced subgraph order \cite{Braun2}.

Turning to consecutive orders, in \cite{mr} and \cite{ir} a link was identified between paths in certain digraphs under the subpath order and the wqo and atomicity problems for words, permutations and equivalence relations under consecutive orders.  These results have many similarities and also some intriguing differences.  In this paper, we investigate the applicability of this methodology in general.  We formulate some conditions -- called \emph{validity} and \emph{bountifulness} -- under which the well quasi-order and atomicity problems can be approached in a uniform manner by studying certain digraphs, yielding various results on how well quasi-order and atomicity are governed (Theorems \ref{thrm wqo bountiful} and \ref{thrm atomicity} and Corollaries \ref{cor atomic bountiful} and \ref{cor nearly there}).  Many well known classes of structures satisfy these properties, such a graphs, digraphs and collections of relations, enabling us to prove decidability of the wqo and atomicity problems for these structures (see Corollaries \ref{cor bountiful egs} and \ref{cor atomic bountiful egs}).  At the end of the paper, we will explore the limitations of these methods, and give an example of structures for which we need to make adjustments to show decidability of the well quasi-order and atomicity problems.

The contents of this paper also appears in the first author's thesis \cite{VIThesis}, and the structure of this paper is as follows.  In Section \ref{sec prelims} we introduce the formal definitions of the concepts mentioned so far, and in Section \ref{sec graph th} we give the necessary ideas and results from graph theory.  Section \ref{sec factor graphs} generalises the \emph{factor graphs} of \cite{mr} and \cite{ir} and gives some related results which will prove useful later.  

Sections \ref{sec valid} and \ref{sec bountiful} introduce \emph{valid} and \emph{bountiful} types of structures respectively.  These notions will allow us to consider the wqo and atomicity problems for a wide class of structures in a uniform way.  Following this, Sections \ref{sec wqo bountiful} and \ref{sec atomicity bountiful} explore the wqo and atomicity problems respectively for bountiful structures.  These investigations yield the main technical results of this paper: decidability of the wqo and atomicity problems for bountiful structures, and a result which offers considerable progress on the atomicity problem for valid structures under the consecutive order (Corollary \ref{cor nearly there}).

In the later sections, we take a different approach, and examine a type of structure which is not valid: a class of permutations called \emph{double ascents}.  In this case, we prove the wqo and atomicity problems to be decidable by relating them to the analogous questions for words.  Section \ref{sec words cons} sets the groundwork for this by recalling McDevitt and Ru\v{s}kuc's investigation into words under the consecutive order \cite{mr} and placing this in the overall framework we have developed.  The final technical sections consider the wqo and atomicity problems for double ascents.  Section \ref{sec double ascents} introduces double ascents and addresses a technicality in situating permutations within the framework we have established.  Section \ref{sec DAs and words} connects double ascents with words over an alphabet of size two, enabling us to tackle the wqo and atomicity problems by applying the results for words in \cite{mr}.  This is achieved in Sections \ref{sec DAs wqo} and \ref{sec DAs atomicity} for well quasi-order and atomicity respectively.

We finish with concluding remarks and a discussion of open problems in Section \ref{sec conclusion}.
 

\section{Preliminaries}
\label{sec prelims}

We will consider relational structures of the form $(X, R)$, where $X$ is a finite set and $R$ is a sequence of relations on $X$.  The sequence of arities of the relations in $R$ forms the \emph{signature} of a structure.  In order to define consecutive orders, we will require $R$ to contain a linear order; without loss of generality we will take this to be the first relation in $R$.  We may view two such relational structures to be of the same kind if they have the same signature and, perhaps, satisfy certain additional conditions.  For example, graphs are relational structures whose signature contains a single symmetric binary relation (plus a linear order, for our purposes).  We will study posets of relational structures of the same kind under the consecutive order.  Given a set $X$ and a partial order $\leq$ on $X$, we will denote the poset of $X$ with $\leq$ by $(X, \leq)$.

Throughout this paper, we will take $\mathbb{N}=\{1, 2, 3, \dots\}$ and $[m, n] = \{m, m+1, \dots, n\}$ for $m, n \in \mathbb{N}$, where $m\leq n$.  We will take the underlying sets of our relational structures to be subsets of $\mathbb{N}$, and the linear order forming the first relation to be the natural one inherited from $\mathbb{N}$.  Two structures of the same kind will be \emph{isomorphic} if, when we relabel the smallest points in their natural linear orders `1', their second smallest `2', and so on, the resulting structures are exactly the same; this will be denoted by $\cong$.  As such, any structure is isomorphic to a unique structure whose underlying set is of the form $[1, n]$ for $n \in \mathbb{N}$.  We consider isomorphic structures to be equal, and we will often use the structure with underlying set $[1, n]$ to represent its class of isomorphic structures.  We denote the restriction of $\rho =(X, R)$ to points in $Y \subseteq X$ by $\rho\Harpoon{Y}$.  The \emph{length} of $(X, R)$ will be $|X|$.  

\begin{ex}
\label{ex intro}
Consider digraphs -- structures whose signature contains our obligatory linear order $\leq$ along with a binary relation.  In a digraph $G=(V_G, \leq, \rho)$ there is a directed edge from $u \in V_G$ to $v \in V_G$ if and only if $(u, v) \in \rho$.  We view the numbers of their underlying sets to label the vertices of digraphs.  Figure \ref{fig intro} shows two digraphs, called $G$ and $H$, along with $G\Harpoon{[2, 3]}$ (reading left to right).  It can be seen that $G \cong H$ and that $|G| = 3$.
\end{ex}

\begin{figure}
\begin{center}
\begin{tikzpicture}[> =  {Stealth [scale=1.3]}, thick]
\tikzstyle{everystate} = [thick]
\node [state, shape = ellipse, minimum size = 10pt] (1) at (0,0) {1};
\node [state, shape = ellipse, minimum size = 10pt] (3) at (2, 0) {3};
\node [state, shape = ellipse, minimum size = 10pt] (2) at (1, 1.5) {2};

\node [state, shape = ellipse, minimum size = 10pt] (4) at (4,0) {4};
\node [state, shape = ellipse, minimum size = 10pt] (9) at (6, 0) {9};
\node [state, shape = ellipse, minimum size = 10pt] (7) at (5, 1.5) {7};

\node [state, shape = ellipse, minimum size = 10pt] (2') at (8,0) {2};
\node [state, shape = ellipse, minimum size = 10pt] (3') at (10, 0) {3};

\path[->]
(1) edge node {} (2)
(1) edge node {} (3)
(4) edge node {} (7)
(4) edge node {} (9)
;
\end{tikzpicture}
\end{center}
\caption{The graphs $G$, $H$ and $G\Harpoon{[2, 3]}$ respectively from Example \ref{ex intro}.}
\label{fig intro}
\end{figure}
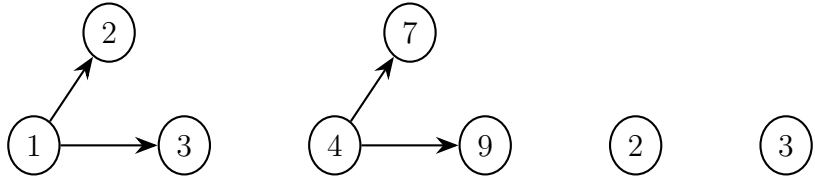

Consecutive orders have been studied for structures such as permutations and words (for example, in \cite{mr}), both of which already have linear orders in their signatures.  In these cases, two structures are related under the consecutive order when there exists an embedding between them which respects the underlying linear orders.  The stipulation that the signatures of our relational structures contain linear orders enables us to extend the concept of consecutive orders to relational structures more generally.

\begin{defn}
Let $X=\{x_1, \dots, x_n\}$, $Y=\{y_1, \dots, y_m\}$ be sets with linear orders $\leq_X$, $\leq_Y$ on them respectively, such that $x_1 \leq_X x_2 \dots \leq_X x_n$ and $y_1 \leq_Y y_2 \leq_Y \dots \leq_Y y_m$.  We will say a mapping $f: X \rightarrow Y$ is \emph{contiguous} with respect to $\leq_X$ and $\leq_Y$ if, given $f(x_1) = y_i$, then $f(x_k) = y_{i+k-1}$ for $k \in [1, n]$.
\end{defn}

\begin{defn}
Let $R=(X, R^{\prime})$ and $S=(Y, S^{\prime})$ be relational structures of the same kind, with underlying linear orders $R_1$ and $S_1$ respectively.  Then $(X, R^{\prime}) \leq (Y, S^{\prime})$ under the \emph{consecutive order} if and only if there exists a contiguous mapping $f:X \rightarrow Y$ with respect to $R_1$ and $S_1$ such that $R \cong S \Harpoon{f(X)}$. 
\end{defn}

Note that all contiguous maps are injective, meaning that, when two structures are related under the consecutive order, one embeds in the other.

We will now introduce the main posets we will consider.

\begin{defn}
\label{defn all posets}
We will denote the posets of each of the following structures under consecutive orders as described here:\\
\\
\begin{tabular}{c c}
\hspace{-2.1cm} Graphs: $\mathcal{G}$; & \hspace{-2.2cm} Permutations: $\mc{P}$;\\
\hspace{-1.1cm} Simple graphs: $\mathcal{S}$; & \hspace{-1.1cm} Equivalence relations: $\mc{E}$;\\
\hspace{-1.8cm} Digraphs: $\mathcal{D}$; &  \hspace{-2.2cm} Linear orders: $\mc{L}$;\\
\hspace{-1.3cm} Tournaments: $\mathcal{T}$; & \hspace{-3cm} Posets: $\mc{PO}$;\\
\hspace{-2.1cm} Forests: $\mathcal{F}$; & \hspace{0.4cm} Words over a finite alphabet $A$: $\mc{W}_A$.\\
\hspace{-7.5mm} Relational structures & \\
with signature $\sigma$: $\mathcal{R}_{\sigma}$;&\\
\end{tabular}
\end{defn}

For most of these structures, it is clear how they can be viewed as relational structures and so how they fit into our overall framework.  Exceptions are perhaps words and permutations.  Here, words are sets with a linear order and a family of unary relations, while permutations may be viewed as relational structures consisting of a set with two linear orders.  We will explain these ideas fully in Sections \ref{sec words cons} and \ref{sec double ascents} respectively.

It can be seen that each of these posets is infinite.  We will be interested in two properties of these posets -- well quasi-order and atomicity.

\begin{defn}
A poset is \emph{well quasi-ordered (wqo)} if it contains neither infinite antichains nor infinite descending sequences.
\end{defn}

For all of the posets we consider, the latter condition holds, as any descending sequence of structures cannot continue beyond structures of length one.  Thus, for our purposes, wqo will be equivalent to the absence of infinite antichains.

\begin{ex} \hspace{7cm}

\begin{itemize}
\item It can immediately be seen that any finite poset is wqo.  

\item A celebrated result of Higman \cite{higman}, known as Higman's Lemma, can be used to show wqo for the poset of words over an alphabet $A$ under the \emph{domination order}. (The domination order is defined as follows: $u_1u_2\dots u_n \leq v_1v_2 \dots v_m$ under the domination order if and only if there is a sequence $i_1 < i_2 < \dots < i_n$ from $[1, m]$ such that $u_j = v_{i_j}$ for each $j$).

\item On the other hand, the \Rev{poset} $\mc{W}_A$ is not wqo whenever $|A| > 1$ -- for example, if $A = \{a, b\}$ then $aa, aba, abba, abbba, \dots$ forms an infinite antichain.
\end{itemize}
\end{ex}

\begin{defn}
A subset $Y \subseteq X$ of a poset $(X, \leq)$ is \emph{downward closed} if $x \in Y$ and $y \leq x$ together imply that $y \in Y$.
\end{defn}

\begin{defn}
An \emph{atomic} set $Y$ is a downward closed subset of a poset $(X, \leq)$ which cannot be written as $Y=U \cup V$ for any two downward closed, proper subsets $U, V$ of $Y$. 
\end{defn}

Atomic sets are also known as ideals, ages and directed sets (for instance, in \cite{fraisse00}).  The following proposition gives an equivalent condition to \Rev{atomicity}, which we will generally use rather than the previous definition; a proof can be found in \cite[Section 2.3.11]{fraisse00}.

\begin{prop}
A subset $Y$ of a poset $(X, \leq)$ is atomic if and only if for any $x, y \in Y$ there exists $z \in Y$ such that $x, y \leq z$; this property is know as the \emph{joint embedding property (JEP)}.  When $x, y \leq z$, we will say that $z$ \emph{joins} $x$ and $y$, or simply that $x$ and $y$ \emph{join}.
\end{prop}

\begin{ex}
The posets $\mc{G}, \mc{S}, \mc{D}$ and $\mc{F}$ are all atomic as they satisfy the JEP: for two graphs $G, H$ in any of these posets, taking $F$ to be the disjoint union of $G$ and $H$, we obtain a graph $F$ such that $G, H \leq F$.  In fact, all of the posets listed in Definition \ref{defn all posets} are atomic.  See Example \ref{ex av sets} for an example of a non-atomic poset.
\end{ex}

\begin{defn}
Given a poset $(X, \leq)$ and $B \subseteq X$, the \emph{avoidance set} of $B$ is the downward closed set of elements which \emph{avoid} $B$:
\begin{equation*}
\Av(B) = \{ x \in X \hspace{2mm}|\hspace{2mm} b \nleq x \text{ for all } b \in B\}.
\end{equation*}
\end{defn}

Downward closed subsets of a poset $(X, \leq)$ may always be expressed as avoidance sets.  To see this, let $Y$ be a downward closed subset of $(X, \leq)$ and note that $Y=\Av(X \backslash Y)$.  Further, since the posets we consider do not contain infinite descending chains, we can express any such subset as $Y=\Av(B)$ where $B$ is the antichain consisting of the minimal elements of $X \backslash Y$;  this choice of $B$ is the unique antichain such that $Y=\Av(B)$.  In this case, $B$ is a \emph{basis} for $Y$; if $B$ is finite, then $Y$ is \emph{finitely based}.  


\begin{ex}
\label{ex av sets}
The avoidance set $\Av(ab, ba)$ of $\mc{W}_{\{a, b\}}$ consists of all of the words containing \Rev{either only $a$'s or only $b$'s}.  It is not atomic because it does not satisfy the JEP -- for instance, there is no word in $\Av(ab, ba)$ containing both $aa$ and $bb$ as subwords.  Indeed, observe that $\Av(ab, ba) = \Av(a) \cup \Av(b)$.
\end{ex}

Finitely based avoidance sets give rise to natural decidability questions: given a finite subset $B$ of a poset, we may ask about decidability of properties of $\Av(B)$.  We study two such questions in this paper: the well quasi-order and atomicity problems.  The \emph{well quasi-order problem} asks if it is decidable, given a finite set $B$, whether $\Av(B)$ is well quasi-ordered?  Similarly, the \emph{atomicity problem} asks if it is decidable, given a finite set $B$, whether $\Av(B)$ is atomic?  

We will be interested in these questions for the posets listed in Definition \ref{defn all posets}.  They have already been answered for $\mc{W}$ in \cite{mr} (and, for wqo, in \cite{atminas}), for $\mc{P}$ in \cite{mr}, and for $\mc{E}$ in \cite{ir}.  In this paper we will answer them for a broad class of posets, which includes $\mc{G, S, D, T}$ and $\mc{R}_{\sigma}$.  To do this, we present a framework that generalises the methods of \cite{mr} and \cite{ir}.  We will also indicate how $\mc{L}$ fits into this overall picture, as well as giving examples which demonstrate the limitations of this approach.  One poset that cannot be studied in this way is $\mc{F}$; in this case, the solution to the wqo problem is presented in the first author's thesis \cite{VIThesis}, while the atomicity problem remains open.  Our methods do not apply to $\mc{PO}$, and so both the wqo and atomicity problems are open for this poset. 

To begin with, we will restrict our considerations to \emph{valid} structures, and in particular to those which have a certain additional property, called \emph{bountiful} structures.  Following this, we look into the wqo and atomicity problems for an example of a structure type which is not valid.

\section{Definitions and results from graph theory}
\label{sec graph th}

In this section we give definitions and results from graph theory which will be needed later to tackle the wqo and atomicity problems for the posets of interest.  Note that this section pertains to standard graphs and digraphs, so there is no assumption of an underlying linear order.

\begin{defn}
A \emph{digraph} is a pair $(V, E)$, where $V$ is a set of vertices and $E \subseteq V \times V$ is a set of directed edges.  If $(u, v) \in E$, we will write $u \rightarrow v$ to indicate that there is an edge from $u$ to $v$.  A digraph is \emph{finite} if $V$ is finite.
\end{defn}

\begin{defn}
A \emph{path} in a digraph is a sequence $v_{1} \rightarrow \dots \rightarrow v_{n}$ of vertices; the number of edges $n-1$ is the path's \emph{length}.  A path is \emph{simple} if all of its vertices are distinct and is a \emph{cycle} if $v_{1} = v_{n}$ and it has length $\geq 1$.  \Rev{A \emph{loop} is a cycle of length 1.}  A cycle is a \emph{simple cycle} if its only repeated vertex is the start/end vertex.  
\end{defn}

\begin{defn}
Let $\pi = u_1 \rightarrow \dots \rightarrow u_n$ and $\eta = v_1 \rightarrow \dots \rightarrow v_m$ be paths in a digraph such that $u_n = v_1$.  The concatenation of $\pi$ and $\eta$ is the path $\pi\eta = u_1 \rightarrow \dots \rightarrow u_n \rightarrow v_2 \rightarrow \dots \rightarrow v_m$.
\end{defn}

\begin{defn}
Given a path $\pi=v_{1} \rightarrow \dots \rightarrow v_{n}$, a \emph{subpath} of $\pi$ is any path $v_{i} \rightarrow \dots \rightarrow v_{j}$ where $1 \leq i \leq j \leq n$.  
\end{defn}

\begin{defn}
A digraph $G=(V, E)$ is \emph{strongly connected} if for any two vertices $u, v \in V$ there is a path in $G$ from $u $ to $v$.
\end{defn}

\begin{defn}
The \emph{in-degree} of a vertex $u$ in a digraph $(V, E)$ is $|\{v \in V : (v, u) \in E\}|$.  The \emph{out-degree} of $u$ is $|\{v \in V : (u, v) \in E\}|$.
\end{defn}

\begin{defn}
A cycle is an \emph{in-out cycle} if at least one vertex has in-degree $>1$ and at least one vertex has out-degree $>1$.
\end{defn}

\begin{defn}
If $\eta, \pi$ are paths in a finite digraph, they are related under the \emph{subpath order} if and only if $\eta$ is a subpath of $\pi$; this is written $\eta \leq \pi$.
\end{defn}

\begin{defn}
Let $G$ be a finite digraph.  A \emph{path complete decomposition} of $G$ is a collection of subgraphs $G_{1}, \dots, G_{n}$ of $G$ such that every path in $G$ is wholly contained in one of the $G_{i}$.
\end{defn}

We now introduce the concept of a \emph{bicycle}, which will be key in identifying when the set of paths of a digraph is wqo or atomic under the subpath order.  

\begin{defn}
A \emph{bicycle} in a digraph consists of two vertex-disjoint, simple cycles connected by a simple path whose internal vertices are disjoint from both cycles.  Either or both cycles may be absent, but if neither cycle is absent then the connecting path must have length at least one.  The first and second cycle will be called the \emph{initial cycle} and \emph{terminal cycle} respectively. 
\end{defn}

\begin{ex}
Figure \ref{fig bicycles} gives examples of different bicycles.
\end{ex}

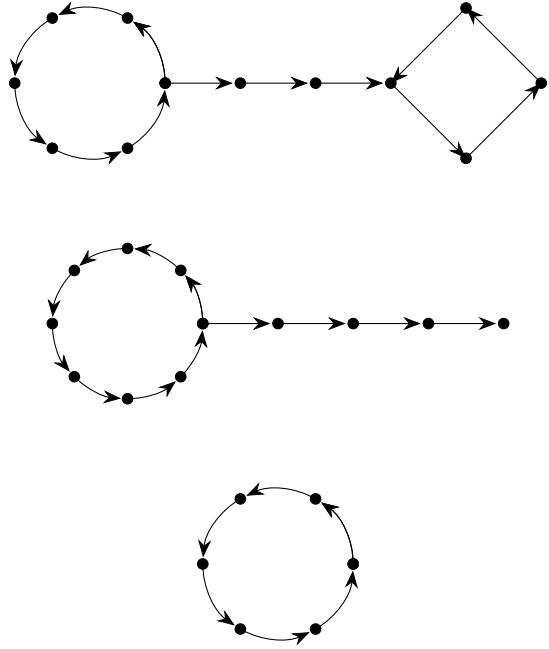
\begin{figure}
\begin{center}

\begin{tikzpicture}
\foreach \x in {0,60,...,360}
   {\draw [-{Stealth [scale=1.4]}] (\x:10mm) arc (\x:\x+55:10mm);}

\foreach \x in {0,60,...,360}
  {\draw [fill] (\x:10mm) circle (0.7mm);}
   
\draw [fill] (20mm,0mm) circle (0.7mm);
\draw [fill] (30mm,0mm) circle (0.7mm);
\draw [fill] (40mm,0mm) circle (0.7mm);

\draw [-{Stealth [scale=1.4]}] (10mm,0mm)--(19mm,0mm);
\draw [-{Stealth [scale=1.4]}] (20mm,0mm)--(29mm,0mm);
\draw [-{Stealth [scale=1.4]}] (30mm,0mm)--(39mm,0mm);

\draw [fill] (60mm,0mm) circle (0.7mm);
\draw [fill] (50mm,-10mm) circle (0.7mm);
\draw [fill] (50mm,10mm) circle (0.7mm);

\draw [-{Stealth [scale=1.4]}] (40mm,0mm)--(50mm,-10mm);
\draw [-{Stealth [scale=1.4]}] (50mm,-10mm)--(60mm,0mm);
\draw [-{Stealth [scale=1.4]}] (60mm,0mm)--(50mm,10mm);
\draw [-{Stealth [scale=1.4]}] (50mm,10mm)--(40mm,0mm);
\end{tikzpicture}

\vspace{1cm}

\begin{tikzpicture}
\foreach \x in {0,45,...,360}
   {\draw [-{Stealth [scale=1.4]}] (\x:10mm) arc (\x:\x+40:10mm);}

\foreach \x in {0,45,...,360}
  {\draw [fill] (\x:10mm) circle (0.7mm);}
   
\draw [fill] (20mm,0mm) circle (0.7mm);
\draw [fill] (30mm,0mm) circle (0.7mm);
\draw [fill] (40mm,0mm) circle (0.7mm);
\draw [fill] (50mm,0mm) circle (0.7mm);

\draw [-{Stealth [scale=1.4]}] (10mm,0mm)--(19mm,0mm);
\draw [-{Stealth [scale=1.4]}] (20mm,0mm)--(29mm,0mm);
\draw [-{Stealth [scale=1.4]}] (30mm,0mm)--(39mm,0mm);
\draw [-{Stealth [scale=1.4]}] (40mm,0mm)--(49mm,0mm);
\end{tikzpicture}

\vspace{1cm}

\begin{tikzpicture}
\foreach \x in {0,60,...,360}
   {\draw [-{Stealth [scale=1.4]}] (\x:10mm) arc (\x:\x+55:10mm);}

\foreach \x in {0,60,...,360}
  {\draw [fill] (\x:10mm) circle (0.7mm);}
   
\end{tikzpicture}

\end{center}
\caption{Various examples of bicycles.}
\label{fig bicycles}

\end{figure}

Our approach to both the wqo and atomicity problems involves encoding structures in our posets as paths in certain digraphs.  The following two results give a means of identifying when the set of paths of a digraph are wqo or atomic under the subpath order.  As such, they will be key tools in relating properties of paths to properties of the structures they encode.  These results are Theorem 3.1 and Theorem 2.1 from \cite{mr}, where they were similarly used in the investigation of wqo and atomicity for words and permutations under consecutive orders.

\begin{prop}
\label{prop digraphs wqo}
If $G$ is a finite digraph, then the following are equivalent:
\begin{enumerate}
\item The set of paths of $G$ is wqo under the subpath order;
\item $G$ contains no in-out cycles;
\item $G$ has a path complete decomposition into bicycles.\qed
\end{enumerate}
\end{prop}

\begin{prop}
\label{prop digraphs atomic}
If $G$ is a finite digraph, the set of paths of $G$ is atomic if and only if $G$ is strongly connected or a bicycle. \qed
\end{prop}

\section{Factor Graphs}
\label{sec factor graphs}

We will now associate the \Rev{finitely based} avoidance sets of our posets with certain digraphs, which will enable us to encode structures as paths in these digraphs.  This will lead to an investigation of the relationship between paths in the digraphs and structures in our avoidance sets, utilising observations such as Propositions \ref{prop digraphs wqo} and \ref{prop digraphs atomic}.  

\Rev{Throughout this paper, we will take $C=\Av(B)$ to be a finitely based avoidance set} and let $b$ be the maximum length of an element of $B$; if $B=\varnothing$, we take $b=1$.  We will let $C_Y=\{\sigma \in C : |\sigma|\in Y\}$ and shorten $C_{\{m\}}$ to $C_m$.

\begin{defn}
Let $m \geq b$.  The \emph{m-dimensional factor graph} $\Gamma_C^m$ of $C=\Av(B)$ is the digraph whose vertices are all structures in $C_m$ and where there is an edge $u \rightarrow v$ if and only if $u\Harpoon{[2, m]} \cong v \Harpoon{[1, m-1]}$.  We will usually work with the $b$-dimensional factor graph, in which case we omit the superscript and just write $\Gamma_C$.
\end{defn}

\Rev{We briefly note that, since factor graphs are digraphs, they can have loops.}

We associate a structure $\sigma \in C_{[m, \infty)}$ of length $n$ with the following path in the $m$-dimensional factor graph:
\begin{equation*}
\Pi(\sigma) = \sigma \Harpoon{[1, m]} \rightarrow \sigma\Harpoon{[2, m+1]} \rightarrow \dots \rightarrow \sigma\Harpoon{[n-m+1, n]}.
\end{equation*}

On the other hand, each path $\pi$ in the $m$-dimensional factor graph will be associated with the set of structures $\Sigma(\pi) =  \{\sigma \in C_{[m, \infty)} \hspace{1mm} | \hspace{1mm} \Pi(\sigma) = \pi\}$. 

In this way, each structure in $C_{[m, \infty)}$ is associated with a path in the $m$-dimensional factor graph (but structures in $C_{[1, m-1]}$ are not).  Conversely, each path in the $m$-dimensional factor graph is associated with zero, one, or several structures in $C_{[m, \infty)}$.  Paths which are associated with several structures will play an important role in relating $C$ to the poset of paths in $\Gamma_C^m$, informing the next definition.

\begin{defn}
A path $\pi$ in $\Gamma_C^m$ is \emph{ambiguous} if $|\Sigma(\pi)| > 1$.
\end{defn}

The following proposition follows directly from the definitions.

\begin{prop}
\label{prop substructures and subpaths}
If $m \geq b$ and $\sigma, \rho \in C_{[m, \infty)}$, then $\sigma \leq \rho$ implies that $\Pi(\sigma) \leq \Pi(\rho)$ in $\Gamma_C^m$.\qed
\end{prop}

\section{Valid structures}
\label{sec valid}

We will now introduce \emph{valid} types of structures, for which it turns out that every path in a factor graph is associated with at least one structure.  Intuitively, for valid types of structures we can combine two overlapping structures to create another structure of the same type. 


\begin{defn}
Let $\sigma, \rho$ be structures of the same type such that $|\sigma|=p$ and $|\rho|=q$.  We say that $\sigma$ and $\rho$ \emph{overlap} if $\sigma\Harpoon{[p-x+1, p]} \cong \rho \Harpoon{[1, x]}$ for some $x \geq 1$.  

Suppose that $\sigma$ and $\rho$ overlap on $x$ points.  Let $\theta$ be a structure of the same type such that $|\theta|=p+q-x$.  We say that $\theta$ \emph{combines} $\sigma$ and $\rho$ if $\theta\Harpoon{[1, p]} \cong \sigma$ and $\theta\Harpoon{[p-x+1, q+p-x]} \cong \rho$.
\end{defn}

\begin{defn}
A type of structure is \emph{valid} if for any two overlapping structures of this type there exists a structure $\theta$ of the same type which combines them.  If a type of structure is not valid, we will say that it is \emph{invalid}. 
\end{defn}

\Rev{We will always be working with posets whose underlying sets consist of structures of a single type.  As such, if this type of structure is valid, we will abuse the terminology slightly by referring to these structures as \emph{valid structures}.}

It is interesting to note that the property of being valid is similar to the amalgamation property (e.g. see \cite{amalgamation}), though validity is a weaker property, as we are only required to be able to combine structures $\sigma, \rho$ that overlap on the `last' points of $\sigma$ and the `first' points of $\rho$, according to the underlying linear order.  

\Rev{Indeed, we would like to thank one of the anonymous reviewers for highlighting the following close relationship between validity and the strong amalgamation property.  Denote by $T$ a type of structure in which there is no interaction between the distinguished linear order and other relations in the signature, and denote by $T^-$ the type of structure obtained by removing the distinguished linear order from $T$.  Any linear ordering of the underlying set of a structure of type $T^-$ produces a structure of type $T$.  We will briefly show that $T$ is valid if and only if $T^-$ satisfies the strong amalgamation property.  

First suppose that $T$ is a valid type of structure.  Take $\sigma$ to embed in $\rho$ and $\tau$ as an amalgam of structures of type $T^-$.  We can add linear orders to $\sigma, \rho$ and $\tau$ so that $\sigma$ embeds in the last $|\sigma|$ elements of $\rho$ and the first $|\sigma|$ elements of $\tau$.  The resulting structures are of type $T$
and so $\rho$ and $\tau$ can be $|\sigma|$-combined by the assumption of validity. Removing the linear order from the resulting structure yields a structure of type $T^-$ which witnesses that $T^-$ satisfies the strong amalgamation property.

For the reverse direction, suppose that $T^-$ satisfies the strong amalgamation property and take structures $\sigma, \rho$ of type $T$ which $x$-overlap for some $x$.  Take $\theta$ to be the restriction of $\sigma$ (and $\rho$) to the points in this overlap.  Let $\sigma^-, \rho^-, \theta^-$ be the structures of type $T^-$ obtained by removing the distinguished linear orders from $\sigma, \rho$ and $\theta$.  Note that $\theta^-$ embeds in $\sigma^-$ and $\rho^-$.  Since $T^-$ satisfies the strong amalgamation property, there is a structure $\tau^-$ such that both $\sigma^-$ and $\rho^-$ embed in $\tau^-$, the images of the copies of $\theta^-$ in $\sigma^-$ and $\rho^-$ coincide, and the intersection of the copies of $\sigma^-$ and $\rho^-$ in $\tau^-$ is precisely this copy of $\theta^-$.  We can add a linear order to $\tau^-$ in such a way as to produce a structure $\tau$ of type $T$ which $x$-combines $\sigma$ and $\rho$, yielding validity of $T$.}\clearpage

\begin{lemma}
\label{lemma valid egs}
The following types of structures are all valid: \\
\\
\begin{tabular}{c c}
\hspace{-2.0cm}(1) Graphs; & \hspace{-1.2cm}(6) Permutations;\\
\hspace{-1.1cm}(2) Simple graphs; & \hspace{0.0cm}(7) Equivalence relations;\\
\hspace{-1.8cm}(3) Digraphs; &  \hspace{-1.1cm}(8) Linear orders;\\
\hspace{-1.3cm}(4) Tournaments; & \hspace{0.8cm} (9) Words over a finite alphabet;\\
\hspace{-0.3cm}(5) \hspace{-0.1cm} Relational structures & \hspace{-2cm}(10) Posets.\\
with signature $\sigma$;&\\
\end{tabular}
\end{lemma}

\begin{proof}
First consider graphs. Let $G$ and $H$ be two graphs on $p, q$ points respectively, and suppose $G\Harpoon{[p-x+1,\hspace{0.7mm} p]} \cong H \Harpoon{[1, x]}$ for some $x \geq 1$.  \Rev{Identifying the last $x$ points of $G$ with the first $x$ points of $H$ uniquely determines the relative values of vertices with respect to the underlying linear order that governs consecutivity.  Doing this, we obtain another graph, so graphs are valid structures.}

By replacing graphs with simple graphs or digraphs in the last paragraph, we see that each of these are also valid structures. 

For tournaments: Take $S$ and $T$ to be two tournaments on $p$ and $q$ points respectively such that $S\Harpoon{[p-x+1,\hspace{0.7mm} p]} \cong T \Harpoon{[1, x]}$ for some $x \geq 1$.  As for graphs, we can identify the last $x$ points of $S$ with the first $x$ points of $T$.  We also need to add a directed edge between each pair of points which are not already neighbours.  The resulting tournament satisfies the conditions for validity.

For relational structures: Let $B=(X, B_1, \dots, B_k), C=(Y, C_1, \dots, C_k)$ be relational structures with signature $\sigma = (n_1, \dots, n_k)$ on $p, q$ points respectively, such that $B\Harpoon{[p-x+1, p]} \cong C \Harpoon{[1, x]}$ for some $x \geq 1$.  Once again, we identify the last $x$ points of $B$ with the first $x$ points of $C$, so the last $x$ points each $B_j$ are identified with the first $x$ points of $C_j$.  The resulting structure $D$ is a relational structure with signature $\sigma$ such that $D\Harpoon{[1, p]} \cong B$ and $D\Harpoon{[p-x+1, \hspace{0.7mm} p+q-x]} \cong C$, so relational structures with the same signature are valid structures.

For permutations: Let $\sigma$ and $\rho$ be two permutations on $p$ and $q$ points respectively such that, as usual, $\sigma \Harpoon{[p-x+1,\hspace{0.7mm} p]} \cong \rho \Harpoon{[1, x]}$ for some $x \geq 1$.  Again, we identify the last $x$ points of $\sigma$ with the first $x$ points of $\rho$.  Doing this uniquely determines the relative values of points with respect to the underlying linear order that governs consecutivity.  It may not uniquely determine the relative positions of points with respect to the other linear order, but by taking any relative positions we obtain a permutation.

For equivalence relations: the process of identifying the overlapping points of two equivalence relations results in identifying the equivalence classes of these points.  This produces another equivalence relation, giving validity.

For linear orders: since of course linear orders contain a linear order naturally in their signatures, we do not need to add an additional linear order.  Hence the only linear orders are of the form $1 \leq 2 \leq \dots \leq n$.  Let $l_1 = 1 \leq 2 \leq \dots \leq n$ and $l_2 = 1 \leq 2 \leq \dots \leq m$ be two such linear orders, with $n \leq m$.  \Rev{It can be seen that $l_1$ and $l_2$ overlap on $x$ points for any $x \leq n$, and that these are the only possible overlaps.  In any of these cases, by identifying the overlapping points we obtain another linear order.}  Hence linear orders are valid.

For words: it is clear that identifying the suffix of one word with the (identical) prefix of another yields a new word.

For posets: Let $\textbf{P}=(P, \leq_p)$ and $\textbf{Q}=(Q, \leq_q)$ be posets on $p, q$ points respectively, such that $\textbf{P}\Harpoon{[p-x+1, \hspace{0.7mm}p]} \cong \textbf{Q} \Harpoon{[1, x]}$ for some $x \geq 1$.  We assume without loss of generality that $P\cap Q$ consists of precisely the $x$ overlapping points.  We will combine and extend $\leq_p$ and $\leq_q$ to form an ordering $\leq_{pq}$ on $P \cup Q$.  We will say that $a \leq_{pq} b$ if and only if one of the following holds: 
\begin{enumerate}
    \item $a, b \in P$ and $a \leq_p b$;
    \item $a, b \in Q$ and $a\leq_q b$;
    \item $a \in P \backslash Q$ and $b \in Q \backslash P$ and there exists $c\in P \cap Q$ such that $a \leq_p c \leq_q b$;
    \item $a \in Q \backslash P$ and $b \in P \backslash Q$ and there exists $c\in P \cap Q$ such that $a \leq_q c \leq_p b$.
\end{enumerate}


We will show that $(P \cup Q, \leq_{pq})$ is a poset.  It is clear that reflexivity is inherited from $\textbf{P}$ and $\textbf{Q}$.  

For transitivity, let $a, b, c \in P \cup Q$ such that $a \leq_{pq} b$ and $b \leq_{pq} c$.  If $a, b, c \in P$ (or Q), then transitivity is inherited from $\textbf{P}$ (or $\textbf{Q}$).  So now consider the case where $a, b \in P \backslash Q$ and $c \in Q \backslash P$.  Here, $a \leq_p b$ and there exists $d \in P \cap Q$ such that $b \leq_p d \leq_q c$.  Therefore, $ a\leq_p b \leq_p d \leq_q c$, meaning that $a \leq_{pq} c$.  Next, consider the case where $a, c \in P \backslash Q$ and $b \in Q \backslash P$.  Since $a \leq_{pq} b$, there exists $d \in P \cap Q$ such that $a \leq_p d \leq_q b$.  Similarly, there exists $e \in P \cap Q$ such that $b \leq_q e \leq_p c$.  We observe that $d \leq_q b \leq_q e$, so $d \leq_q e$ by transitivity of $\leq_q$.  Since $\textbf{P}$ and $\textbf{Q}$ are identical on the points of $P \cap Q$, it follows that $d \leq_p e$ as well.  Therefore, $a \leq_p d \leq_p e \leq_p c$, and so $a \leq_p c$ by transitivity of $\textbf{P}$.  It follows that $a \leq_{pq} c$.  \Rev{The other cases are proven similarly.}

Finally, we will prove that $\leq_{pq}$ is antisymmetric.  Suppose that $a, b \in P \cup Q$ such that $a \leq_{pq} b$ and $b \leq_{pq} a$.  If $a, b \in P$ (or $Q$), then $a=b$ by antisymmetry of $\leq_p$ (or $\leq_q$).  Otherwise, suppose $a \in P\backslash Q$ and $b \in Q\backslash P$.  There exist $c, d \in P \cap Q$ such that $a \leq_p c \leq_q b$ and $b \leq_q d \leq_p a$.  Therefore, $d \leq_p a \leq_p c$, meaning that $d \leq_p c$.  Similarly, $ c\leq_q b \leq_q d$, so $c \leq_q d$.  Since $c, d \in P \cap Q$, their relationships in $\textbf{P}$ and $\textbf{Q}$ are identical by assumption.  Therefore, $c = d$. It follows that $d=c=b=a$, completing the proof.  
\end{proof}

Recall that paths in factor graphs are associated with zero, one or several structures.  Valid structures are significant because in their factor graphs every path is associated with a non-empty set of structures, as proved in the next lemma.

\begin{lemma}
\label{lemma validity 2}
Let $(X, \leq)$ be a poset of valid structures under the consecutive order and $C=\Av(B)$ be an avoidance set of $(X, \leq)$.  Then for any $m \geq b$, any path $\pi$ in $\Gamma_C^m$ satisfies $|\Sigma(\pi)| \geq 1$. 
\end{lemma}

\begin{proof}
Consider an avoidance set $C=\Av(B)$ of $(X, \leq)$ and let $m \geq b$.  Take a path $\pi$ in $\Gamma_C^m$.   We will use induction on the length of $\pi$.  Suppose the length of $\pi$ is 0, say $\pi = v_0$.  Since $v_0$ is a vertex, it is an element of $C$ by definition, so $v_0 \in \Sigma(\pi)$ and $|\Sigma(\pi)| \geq 1$.  Now suppose the length of $\pi$ is 1, say $\pi = v_0 \rightarrow v_1$.  This means that $v_0 \Harpoon{[2, \hspace{0.7mm}\Rev{m}]} \cong v_1 \Harpoon{[1, \hspace{0.7mm} \Rev{m}-1]}$, and so by the definition of validity there exists a structure $v$ on $\Rev{m}+1$ points such that $v \Harpoon{[1, \hspace{0.7mm}\Rev{m}]} \cong v_0$ and $v \Harpoon{[2,\hspace{0.7mm} \Rev{m}+1]} \cong v_1$.  This structure $v$ has associated path $\pi$, meaning that $|\Sigma(\pi)| \geq 1$. \Rev{To see that $v \in C$, start by assuming for a contradiction that $\sigma \leq v$ for some $\sigma \in B$.  Since $\sigma \nleq v_0, v_1$, it follows that the embedding of $\sigma$ in $v$ must include both the first point of $v_0$ and the last point of $v_1$.  Therefore, $\sigma=v$, since the embedding is contiguous.  However, $|\sigma|=|v|=b+1>b$, so $\sigma \notin B$, a contradiction.}

Now suppose all paths of length $<k$ in $\Gamma_C^m$ have $\geq 1$ associated structure.  Let $\pi = v_0 \rightarrow \dots \rightarrow v_k$ be a path of length $k$ in $\Gamma_C^m$.  By assumption the paths $v_0 \rightarrow \dots \rightarrow v_{k-1}$ and $v_{k-1} \rightarrow v_k$ each have at least one associated structure; call these $\rho$ and $\sigma$ respectively.  Now $\rho$ and $\sigma$ overlap on the $m$ points of $v_{\Rev{k-1}}$, and so since the structures of this poset are valid there exists a structure $\theta$ which combines $\sigma$ and $\rho$.  \Rev{It is clear that $\Pi(\theta) = \pi$, so $|\Sigma(\pi)| \geq 1$.  That $\theta\in C$ can be shown using the reasoning given for the base case.}
\end{proof}

The following example demonstrates a very natural instance of a poset which is not valid: the poset $\mathcal{F}$ of forests under the consecutive order.  

\begin{ex}
\label{ex factor graphs don't work}
\Rev{We will show that forests are not a valid type of structures.  Consider the following two forests:

\begin{center}
\begin{tikzpicture}[> =  {Stealth [scale=1.3]}, thick]
\tikzstyle{everystate} = [thick]
\node [state, shape = ellipse, minimum size = 10pt] (1) at (0,0) {1};
\node [state, shape = ellipse, minimum size = 10pt] (3) at (2, 0) {3};
\node [state, shape = ellipse, minimum size = 10pt] (2) at (1, 1.5) {2};
\node (A) at (-1, 0.75) {$G \h =$};
\path[-]
(1) edge node {} (2)
(1) edge node {} (3);

\node [state, shape = ellipse, minimum size = 10pt] (1a) at (5.5,0) {1};
\node [state, shape = ellipse, minimum size = 10pt] (3a) at (7.5, 0) {3};
\node [state, shape = ellipse, minimum size = 10pt] (2a) at (6.5, 1.5) {2};
\node (B) at (4.5, 0.75) {$H \h =$};
\path[-]
(3a) edge node {} (2a)
(1a) edge node {} (3a);
\end{tikzpicture}
\end{center}

It can be seen that $G \Harpoon{[2, 3]} \cong H \Harpoon{[1, 2]}$.  However, any graph which combines $G$ and $H$ contains the cycle shown in Figure \ref{fig cycle} as a subgraph, and so is certainly not a forest.  Hence forests are not valid structures.

We will now consider how this affects factor graphs.  Consider the factor graph in Figure \ref{fig dodgy factor graph}, which is the factor graph of the avoidance set $C$ of $\mathcal{F}$ which avoids all forests on three points except for $G$ and $H$. One graph associated with the path $\pi = G \rightarrow H$ is that shown in Figure \ref{fig cycle}.  Since this is a cycle, it is not a forest, and therefore not in $C$.  Every graph associated with $\pi$ contains this cycle as a subgraph, so $\pi$ has no associated forests.  This means that $\Sigma(\pi) = \varnothing$ and so we can also see that $\mathcal{F}$ is not valid using the contrapositive of Lemma \ref{lemma validity 2}.}
\end{ex}

\begin{figure}
\begin{center}
\begin{tikzpicture}[> =  {Stealth [scale=1.3]}, thick]
\tikzstyle{everystate} = [thick]
\node [state, shape = ellipse, minimum size = 10pt] (1) at (0,0) {1};
\node [state, shape = ellipse, minimum size = 10pt] (3) at (2, 0) {3};
\node [state, shape = ellipse, minimum size = 10pt] (2) at (1, 1.5) {2};
\node [state, shape = ellipse, minimum size = 10pt] (1a) at (5,0) {1};
\node [state, shape = ellipse, minimum size = 10pt] (3a) at (7, 0) {3};
\node [state, shape = ellipse, minimum size = 10pt] (2a) at (6, 1.5) {2};

\path[-]
(1) edge node {} (2)
(1) edge node {} (3)
(1a) edge node {} (3a)
(2a) edge node {} (3a)
;

\draw [-{Stealth [scale=1.4]}, color=black] (3,1)--(4,1);
\draw [-{Stealth [scale=1.4]}, color=black] (4,0.2)--(3,0.2);

\draw (1,0.6) circle (2cm);
\draw (6,0.6) circle (2cm);
\end{tikzpicture}
\end{center}
\caption{The factor graph from Example \ref{ex factor graphs don't work} whose paths can create cycles.}
\label{fig dodgy factor graph}
\end{figure}
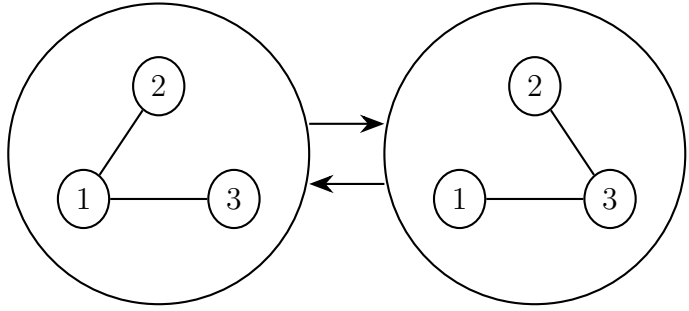

\begin{figure}
\begin{center}
\begin{tikzpicture}[> =  {Stealth [scale=1.3]}, thick]
\tikzstyle{everystate} = [thick]
\node [state, shape = ellipse, minimum size = 10pt] (1) at (0,0) {1};
\node [state, shape = ellipse, minimum size = 10pt] (3) at (2, 0) {3};
\node [state, shape = ellipse, minimum size = 10pt] (2) at (1, 1.5) {2};
\node [state, shape = ellipse, minimum size = 10pt] (4) at (3, 1.5) {4};

\path[-]
(1) edge node {} (2)
(1) edge node {} (3)
(2) edge node {} (4)
(3) edge node {} (4)
;
\end{tikzpicture}
\end{center}
\caption{A graph corresponding to the path $\pi$ in Example \ref{ex factor graphs don't work}}
\label{fig cycle}
\end{figure}
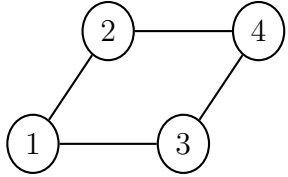

\begin{rem}
    We note that structures which can be defined as downward closed subsets of valid structures are not necessarily valid themselves.  \Rev{One example of this is given by forests: even though they are a downward closed subset of a valid type of structure (namely, graphs), we have seen in Example \ref{ex factor graphs don't work} that they are not valid.}  \Rev{Another} example is given by overlap free words.  A word over an alphabet $A$ is \emph{overlap free} if it does not contain any subword of the form $xyxyx$ for \Rev{$x \in A, y \in A \cup \{\epsilon\}$.  Note that this definition is different to the one commonly used in the literature (eg. \cite{overlapref}), where $y$ can be a word.}  To see that overlap free words are not valid structures, consider the set of overlap free words over $A=\{a, b\}$. Take $u = aba$, $v = baba$; clearly $u, v \in A^*$. Moreover, $u \Harpoon{[2,3]}= ba = v \Harpoon{[1,2]}$. However, identifying these two letters yields the word $ababa$, which is not overlap free.

In this case we can use the methodology for words under the consecutive order to answer the wqo and atomicity problems for overlap free words under the consecutive order.  Let $\mc{OA}_A$ be the poset of overlap free words over an alphabet A under the consecutive order. We consider an avoidance set $C = \Av(B)$ of $\mc{OA}_A$.  Take $Y = \{xyxyx \hspace{2mm}| \hspace{2mm} x \in A, y \in \Rev{A \cup \{\epsilon\}}\}$. It can be seen that $C$ is equal to the avoidance set $\Av(B \cup Y )$ of $\mc{W}_A$.  Hence wqo and atomicity for $C$ are governed by wqo and atomicity respectively for this avoidance set of $\mc{W}_A$.
\end{rem}

In the final sections of this paper we will exhibit a further example of a structure type which is not valid -- permutations consisting of at most two ascending sequences -- and solve the wqo and atomicity problems for these structures by adapting our techniques for valid structures.

We will finish this section with a couple of useful observations about posets of valid structures, beginning with the fact that in these posets the absence of ambiguous paths yields the converse of Proposition \ref{prop substructures and subpaths}.

\begin{lemma}
\label{lemma 1-1}
Let $C=\Av(B)$ be an avoidance set of a poset of valid structures, $m \geq b$ and $\sigma, \rho \in C_{[m, \infty)}$.  If $\Gamma_C^m$ contains no ambiguous paths, then $\sigma \leq \rho$ if and only if $\Pi(\sigma) \leq \Pi(\rho)$.
\end{lemma}

\begin{proof}
$(\Rightarrow)$ This is Proposition \ref{prop substructures and subpaths}.
$(\Leftarrow)$ If $\Pi(\sigma) \leq \Pi(\rho)$, then since there are no ambiguous paths, $\Pi(\sigma)$ is only associated with $\sigma$.  Similarly, the subpath of $\Pi(\sigma)$ corresponding to $\Pi(\rho)$ is only associated with $\rho$.  Therefore, since we can extend $\Pi(\rho)$ to $\Pi(\sigma)$, it must be the case that $\rho$ is a substructure of $\sigma$.
\end{proof}

\begin{prop}
\label{prop paths}
Let $C$ be an avoidance set of a poset of valid structures.  If $\pi=\eta\xi$ in $\Gamma_C$, then for any $\rho \in \Sigma(\eta)$ and $\nu \in \Sigma(\xi)$ there exists $\sigma \in \Sigma(\pi)$ such that $\sigma \Harpoon{[1, \hspace{0.7mm}|\rho|]} \cong \rho$ and $\sigma \Harpoon{[|\rho|-b+1, \hspace{0.7mm}|\rho|+|\nu|-b]} \cong \nu$, so $\rho, \nu \leq \sigma$.
\end{prop}

\begin{proof}
Suppose $|\rho|=p$ and $|\nu|=q$.  Since $\pi=\eta\xi$, the paths $\eta$ and $\xi$ overlap on a vertex, meaning that $\rho \Harpoon{[p-b+1, \hspace{0.7mm}p]} \cong \nu \Harpoon{[1,\hspace{0.7mm} b]}$.  By assumption, the underlying structures are valid, so by definition there exists a structure $\sigma$ on $p+q-b$ points such that $\sigma \Harpoon{[1,\hspace{0.7mm} p]} \cong \rho$ and $\sigma \Harpoon{[p-b+1,\hspace{0.7mm} p+q-b]} \cong \nu$.  It can be seen that $\sigma \in \Sigma(\pi)$ as required.
\end{proof}

\section{Bountiful structures}
\label{sec bountiful}

This section introduces some special classes of structures, for which all paths in factor graphs are ambiguous.  We will answer the well quasi-order and atomicity problems for posets of these structures under the consecutive order in the affirmative, meaning that these questions are decidable for a wide range of structures, including graphs and digraphs. \clearpage

\begin{defn}
\label{defn bountiful}
\Rev{A valid type of structure is said to be \emph{bountiful} if for any structures $\sigma, \rho$ of this type which satisfy
\begin{itemize}
    \item $|\sigma|=|\rho|=m$ for some $m \in \N$;
    \item $\sigma \Harpoon{[2, m]} \cong \rho \Harpoon{[1, m-1]}$;
\end{itemize}

there exist at least two structures $\theta_1, \theta_2$ of the same type and of length $m+1$, such that $\theta_i\Harpoon{[1, \hspace{0.7mm}m]} \cong \sigma$ and $\theta_i\Harpoon{[2,\hspace{0.7mm} m+1]} \cong \rho$ for $i=1, 2$.}

\end{defn}

\Rev{As for valid structures, if the type of structure of the underlying set of a poset is bountiful, we will abuse the terminology by referring to structures of this type as \emph{bountiful structures}.}

The next lemma shows that every non-trivial path is ambiguous in factor graphs of posets of bountiful structures.

\begin{lemma}
\label{lemma tfae}
Let $C=\Av(B)$ be an avoidance set of a poset of bountiful structures under the consecutive order and $m \geq b$.  Then the following hold:
\begin{enumerate}
\item Any path $\pi$ of length 1 in $\Gamma_C^m$ is ambiguous $(\Sigma(\pi) \geq 2)$;
\item All paths of length $\geq 1$ in $\Gamma_C^m$ are ambiguous.
\end{enumerate}
\end{lemma}

\begin{proof}
We begin by showing that $(1)$ holds.  Given any path $\pi = \sigma \rightarrow \rho$ in $\Gamma^m_C$, $\sigma \Harpoon{[2,\hspace{0.7mm} m]} \cong \rho \Harpoon{[1,\hspace{0.7mm} m-1]}$ so there are two structures $\theta_1, \theta_2$ on $m+1$ points such that $\theta_i\Harpoon{[1,\hspace{0.7mm} m]} \cong \sigma$ and $\theta_i\Harpoon{[2, \hspace{0.7mm}m+1]} \cong \rho$, for $i=1, 2$, by the assumption of bountifulness.  \Rev{That $\theta_1, \theta_2 \in C$ follows by the same argument used in the proof of Lemma \ref{lemma validity 2}.}  Both $\theta_1$ and $\theta_2$ are in $\Sigma(\pi)$, so $|\Sigma(\pi)|\geq 2$ as required.

\Rev{We will use the fact that $(1)$ holds to prove $(2)$.  Let $\pi=v_1 \rightarrow \dots \rightarrow v_n$ be a path in $\Gamma_C^m$.  If $|\pi|=1$ then $\pi$ is ambiguous by $(1)$, so suppose that $|\pi|>1$ (meaning that $n \geq 3$).  By $(1)$, the subpath $\eta = v_1 \rightarrow v_2$ is ambiguous, so there exist two distinct structures $\theta_1, \theta_2 \in \Sigma(\eta)$.  Consider the subpath $\pi^-=v_2 \rightarrow \dots \rightarrow v_n$ of $\pi$.  By Lemma \ref{lemma validity 2}, there exists $\sigma \in \Sigma(\pi^-)$.  Now $\sigma$ overlaps with both $\theta_1$ and $\theta_2$ on the $m$ vertices of $v_2$, so since bountiful types of structures are also valid, there exist structures $\rho_1, \rho_2$ which combine $\sigma$ and $\theta_1, \theta_2$ respectively.  Note that $\rho_1$ and $\rho_2$ are distinct because $\theta_1$ and $\theta_2$ are distinct.  Again, that $\rho_1, \rho_2 \in C$ follows in the same way as in the proof of Lemma \ref{lemma validity 2}.  To finish, we observe that, since $\Pi(\rho_1)=\Pi(\rho_2)=\pi$, it follows that $\rho_1, \rho_2 \in \Sigma(\pi)$, and so $\pi$ is ambiguous.}
\end{proof}


\begin{lemma}
\label{lemma bountiful structures}
Each of the following are bountiful structures: \\
(1) Graphs \hspace{1.5cm} (2) Simple graphs \hspace{1.5cm} (3) Digraphs\\ (4) Tournaments \hspace{0.7cm} (5) Relational structures with signature $\sigma$.
\end{lemma}

\begin{proof}
In each case, we show that the structures satisfy Definition \ref{defn bountiful}.  

(1) Let $G, H$ be graphs on $m$ points such that $G \Harpoon{[2, \hspace{0.7mm} m]} \cong H \Harpoon{[1, \hspace{0.7mm}m-1]}$.  We can obtain two different graphs $W_1, W_2$ on $m+1$ points so that $W_i\Harpoon{[1, \hspace{0.7mm}m]} \cong G$ and $W_i\Harpoon{[2,\hspace{0.7mm} m+1]} \cong H$ (for $i=1, 2$) by either including or excluding an edge from $1$ to $m+1$.

(2) For simple graphs, we do exactly the same as for graphs.

(3) For any digraphs $D, E$ on $m$ points such that $D \Harpoon{[2,\hspace{0.7mm} m]} \cong E \Harpoon{[1, \hspace{0.7mm}m-1]}$, we can find two different digraphs $F_1, F_2$ on $m+1$ points so that $F_i\Harpoon{[1,\hspace{0.7mm} m]} \cong D$ and $F_i\Harpoon{[2, \hspace{0.7mm}m+1]} \cong E$ (for $i=1, 2$) by either including or excluding a directed edge from $1$ to $m+1$.

(4) For any two tournaments $S, T$ on $m$ points such that $S \Harpoon{[2,\hspace{0.7mm} m]} \cong T \Harpoon{[1, \hspace{0.7mm}m-1]}$, we obtain two different tournaments $U_1, U_2$ on $m+1$ points so that $U_i\Harpoon{[1,\hspace{0.7mm} m]} \cong S$ and $U_i\Harpoon{[2,\hspace{0.7mm} m+1]} \cong T$ (for $i=1, 2$) by either including a directed edge from $1$ to $m+1$ or from $m+1$ to 1.

(5) For any relational structures $R, S$ with signature $\sigma$ on $m$ points such that $R \Harpoon{[2, \hspace{0.7mm}m]} \cong S \Harpoon{[1,\hspace{0.7mm} m-1]}$, we can construct two different relational structures $C_1, C_2$ with signature $\sigma$ on $m+1$ points so that $C_i\Harpoon{[1,\hspace{0.7mm} m]} \cong R$ and $C_i\Harpoon{[2, \hspace{0.7mm}m+1]} \cong S$ (for $i=1, 2$) by either including or excluding a relation including both 1 and $m+1$ in any component.
\end{proof}

\begin{rem}
\label{lemma not bountiful structures}
The following structures are not bountiful: \\
(1) Forests; \\
(2) Permutations;\\
(3) Equivalence relations;\\
(4) Linear orders;\\
(5) Posets;\\
(6) Words over a finite alphabet.
\end{rem}

\begin{proof}
(1) Forests are not bountiful because they are not valid, as shown in Example \ref{ex factor graphs don't work}.  For (2) and (6), see \cite{mr}, and for (3), see \cite{ir}, for examples showing paths which have only one associated structure, breaking condition $(2)$ of Lemma \ref{lemma tfae}.  

We will look at linear orders in a little more detail.  As noted in the proof of Lemma \ref{lemma valid egs}, in our framework the only linear order on $\Rev{m}$ points is $1 \leq 2 \leq \dots \leq m$.  Therefore, for linear orders, $m$-dimensional factor graphs consist of a single vertex of this form with a loop on it.  The only structure associated with a path of length one in such a factor graph is $1 \leq 2 \leq \dots \leq m \leq m+1$, meaning that condition $(1)$ of Lemma \ref{lemma tfae} is broken.  Hence linear orders are not bountiful.

For (5), we consider two specific, identical posets: linear orders on $m$ points, which of course overlap on $m-1$ points.  Identifying these points yields exactly one poset: the linear order on $m+1$ points.  Therefore, posets are not bountiful.
\end{proof}

\section{Well quasi-order for bountiful structures}
\label{sec wqo bountiful}

In this section we establish criteria for wqo of avoidance sets of bountiful structures.  We begin with two general results which show that two cycle types lead to non-wqo. These results mirror analogous results in \cite{mr} and \cite{ir}.

\begin{lemma}
\label{lemma in out cycle}
Let $C$ be an avoidance set of a poset of valid structures under the consecutive order.  If $\Gamma_C^m$ contains an in-out cycle, then $C$ is not wqo.
\end{lemma}

\begin{proof}
Since $\Gamma_C^m$ contains an in-out cycle, its paths are not wqo under the subpath order by Proposition \ref{prop digraphs wqo}.  That $C$ is not wqo follows by the contrapositive of Proposition \ref{prop substructures and subpaths}.
\end{proof}

\begin{lemma}
\label{lemma amb cycles}
Let $C$ be an avoidance set of a poset of valid structures under the consecutive order.  If $\Gamma_C$ contains an ambiguous cycle, $C$ is not wqo.
\end{lemma}

\begin{proof}
We use the same method as given for $\mc{P}$ by McDevitt and Ru\v{s}kuc in \cite{mr}.

Suppose $\eta =\sigma_{1} \rightarrow \sigma_{2} \rightarrow \dots \rightarrow \sigma_{n}$ is the ambiguous cycle and take $k$ to be the length of the shortest ambiguous subpath of $\eta$.  Now we consider all subpaths of $\eta$ of length $k$, so for each $i \in [1, n]$, we take $\eta_{i} = \sigma_{i} \rightarrow \sigma_{i+1} \rightarrow \dots \rightarrow \sigma_{i+k}$, where $\sigma_{j} = \sigma_{j-n}$ for $j>n$.  For each such $i$, let $\rho_{i} \in \Sigma(\eta_{i})$.  Suppose that $\eta_{m}$ is a (minimal) ambiguous subpath of $\eta$.

Since $\rho_{i}\Harpoon{[2, \hspace{0.7mm}b+k]}$ and $\rho_{i+1}\Harpoon{[1,\hspace{0.7mm} b+k-1]}$ trace the same path in $\Gamma_{C}$ for each $i$, $\rho_{i} \rightarrow \rho_{i+1}$ is an edge in the $b+k$-dimensional factor graph of $C$.  This means that $\pi = \rho_{1} \rightarrow \rho_{2} \rightarrow \dots \rightarrow \rho_{n} \rightarrow \rho_{1}$ is a cycle in this graph.

Since $\eta_{m}$ is ambiguous, it has at least two associated structures $\rho_{m}, \rho_{m}^{\prime} \in \Sigma(\eta_{m})$, and so by the previous discussion $\rho_{m}^{\prime} \rightarrow \rho_{m+1}$ is an in-edge to $\pi$ and $\rho_{m-1} \rightarrow \rho^{\prime}_{m}$ is an out-edge to $\pi$.  Thus $\pi$ is an in-out cycle, and so $C$ is not wqo by Lemma \ref{lemma in out cycle}.  
\end{proof}

We are now ready to prove our main theorem on well quasi-order, which provides criteria for avoidance sets of posets of bountiful structures to be wqo under the consecutive order.  Perhaps surprisingly, this result shows that the only such wqo avoidance sets are the finite ones.

\begin{thrm}
\label{thrm wqo bountiful}
Let $C=\Av(B)$ be a \Rev{finitely based} avoidance set of a poset of bountiful structures.  The following are equivalent:
\begin{enumerate}
\item $C$ is wqo under the consecutive order;
\item $C$ is finite;
\item $\Gamma_C$ contains no cycles.
\end{enumerate}
\end{thrm}

\begin{proof}
In what follows, recall that if $B=\varnothing$, then $b=1$ and so $\Gamma_C$ is not empty. \Rev{If $\Gamma_C$ is empty, the result is trivial, so we assume that this is not the case for the remainder of this proof.}  

$(1) \Rightarrow (3)$ We prove the contrapositive.  If $\Gamma_C$ contains a cycle, it is ambiguous by Lemma \ref{lemma tfae}.  It follows that $C$ is not wqo by Lemma \ref{lemma amb cycles}.  

$(3) \Rightarrow (2)$ Suppose $\Gamma_C$ has no cycles and $n$ vertices.  The maximum length of a path is $n-1$, so there is a bound on the length of structures associated with paths of $\Gamma_C$.  Hence, there is a bound on the length of structures in $C$, so $C$ is finite.

$(2) \Rightarrow (1)$ If $C$ is finite, it is wqo as it cannot contain infinite antichains.
\end{proof}

Since it is easy to construct $\Gamma_C$ and check whether it contains cycles, we obtain the following corollary:

\begin{cor}
\label{cor bountiful}
The wqo problem is decidable for posets of bountiful structures under the consecutive order. \qed
\end{cor}

\begin{cor}
\label{cor bountiful egs}
In each of the following posets, a \Rev{finitely based} avoidance set is wqo if and only if it is finite:
\begin{center}
(1) $\mathcal{G}$; \hspace{1.5cm} (2) $\mathcal{S}$; \hspace{1.5cm} (3) $\mathcal{D}$;\\
\hspace{-2.4cm}(4) $\mathcal{T}$; \hspace{1.5cm} (5) $\mathcal{R}_{\sigma}$.
\end{center}
\end{cor}

\begin{proof}
All of these are posets of bountiful structures by Lemma \ref{lemma bountiful structures}, and so the result follows from Theorem \ref{thrm wqo bountiful}.
\end{proof}

\section{Atomicity for Bountiful Structures}
\label{sec atomicity bountiful}

In this section, we turn to the atomicity problem for posets of valid structures and bountiful structures.  Our first theorem gives criteria for an avoidance set of a poset of valid structures to be atomic; this is a generalisation of the results given for $\mc{W}$ and $\mc{P}$ in \cite{mr} and for $\mc{E}$ in \cite{ir}.  Since all bountiful structures are valid, we use this theorem to prove decidability of atomicity for posets of bountiful structures.  

\begin{thrm}
\label{thrm atomicity}
Let $C=\Av(B)$ be \Rev{an infinite, finitely based} avoidance set of a poset of valid structures under the consecutive order.  Then $C$ is atomic if and only if:
\begin{enumerate}
\item $\Gamma_C$ is strongly connected or a bicycle with no ambiguous paths; and
\item For each $\sigma \in C_{[1, b-1]}$ there is $\rho \in C_{b}$ such that $\sigma \leq \rho$.
\end{enumerate}
\end{thrm}

\begin{proof}
$(\Leftarrow)$ Take $\sigma, \rho \in C$ and extend them to $\sigma^{\prime}, \rho^{\prime} \in C_{b}$ respectively, using (2) if necessary.  

If $\Gamma_C$ is strongly connected, there is a path $\eta$ from $\Pi(\sigma^{\prime})$ to $\Pi(\rho^{\prime})$.  Let $\pi=\Pi(\sigma^{\prime})\hspace{0.3mm}\eta\hspace{0.5mm}\Pi(\rho^{\prime})$.  By Proposition \ref{prop paths}, there exists $\tau \in \Sigma(\pi)$ such that $\sigma^{\prime}, \rho^{\prime} \leq \tau$, so $\sigma, \rho \leq \tau$.  Therefore $C$ satisfies the JEP and so is atomic.

If $\Gamma_C$ is not strongly connected, it is a bicycle with no ambiguous paths.  Since there are no ambiguous paths, by Lemma \ref{lemma 1-1}, $\Gamma_C$ is atomic if and only if $C_{[b, \infty)}$ is atomic.  As $\Gamma_C$ is a bicycle, it is atomic by Proposition \ref{prop digraphs atomic}, meaning that $C_{[b, \infty)}$ is also atomic.  Now, since (2) holds, we can conclude that $C$ is atomic.

$(\Rightarrow)$ Suppose $C$ is atomic.  We will show that (2) holds using the contrapositive.  Assume that there is $\sigma \in C_{[1, b-1]}$ such that $\sigma \nleq \rho$ for all $\rho \in C_b$.  Take $\tau \in C_{[b, \infty)}$.  Since $C$ is atomic, there exists $\theta \in C$ such that $\sigma, \tau \leq \theta$.  Since $\tau \leq \theta$, $|\theta| \geq b$, so there is a subsequence of $\theta$ of length $b$ which contains $\sigma$, yielding a contradiction.

Now we turn to show that (1) holds.  First, suppose that $\Gamma_C$ is not atomic, so there exist paths $\pi, \eta$ which do not join.  Take $\sigma \in \Sigma(\pi)$ and $\rho \in \Sigma(\eta)$.  Since $C$ is atomic, there exists $\theta \in C$ such that $\sigma, \rho \leq \theta$.  Then $\pi, \eta \leq \Pi(\theta)$ by Proposition \ref{prop substructures and subpaths}, a contradiction.  Therefore $\Gamma_C$ must be atomic, so it is strongly connected or a bicycle by Proposition \ref{prop digraphs atomic}.

Suppose that $\Gamma_C$ is a bicycle but is not strongly connected and, aiming for a contradiction, suppose that it has an ambiguous path $\pi$.  \Rev{By adding new vertices to the start or end of $\pi$}, we can extend $\pi$ to an ambiguous path $\pi^{\prime}$ which enters more than one strongly connected component.  Take two distinct structures $\sigma_1, \sigma_2 \in \pi^{\prime}$.  Since $C$ is atomic, there exists $\theta \in C$ such that $\sigma_1, \sigma_2 \leq \theta$, meaning that $\Pi(\sigma_1), \Pi(\sigma_2) \leq \Pi(\theta)$ by Proposition \ref{prop substructures and subpaths}.  Since $\sigma_1, \sigma_2$ are distinct, they cannot embed in $\theta$ via the same embedding.  Therefore, they correspond to different subsets of the points of $\theta$, so their paths are different subpaths of $\Pi(\theta)$.  In other words, $\Pi(\theta)$ has two distinct subpaths which are both isomorphic to $\pi^{\prime}$.  This means that $\Pi(\theta)$ traverses every edge of $\pi^{\prime}$ twice, contradicting the assumption that $\pi^{\prime}$ enters more than one strongly connected component.  Hence $\Gamma_C$ cannot have ambiguous paths.
\end{proof}

\begin{cor}
\label{cor atomic bountiful}
Let $C=\Av(B)$ be \Rev{an infinite, finitely based} avoidance set of a poset of bountiful structures under the consecutive order.  Then $C$ is atomic if and only if:
\begin{enumerate}
\item $\Gamma_C$ is strongly connected; and
\item For each $\sigma \in C_{[1, b-1]}$ there is $\rho \in C_{b}$ such that $\sigma \leq \rho$.
\end{enumerate}
\end{cor}

\begin{proof}
$(\Rightarrow)$ Since bountiful posets are valid, if $C$ is atomic then by Theorem \ref{thrm atomicity}, $(2)$ holds and $\Gamma_{\Rev{C}}$ is strongly connected or a bicycle with no ambiguous paths.  Since $C$ is bountiful, all paths in $\Gamma_{\Rev{C}}$ are ambiguous, so $\Gamma_{\Rev{C}}$ must be strongly connected, giving $(1)$.

$(\Leftarrow)$ This is given by Theorem \ref{thrm atomicity}.
\end{proof}

\begin{cor}
The atomicity problem is decidable for posets of bountiful structures under the consecutive order.
\end{cor}

\begin{proof}
\Rev{It is easy to check whether a finitely based avoidance set is finite, as we can just check whether its factor graph contains cycles.  For the finite avoidance sets, we can algorithmically check for atomicity.  For infinite avoidance sets,} consider the conditions of Corollary \ref{cor atomic bountiful}. It is easy construct $\Gamma_C$ and check whether it is strongly connected.  For the second condition, there are only finitely many elements in $C_{[1, b-1]}$ and finitely many elements in $C_b$ that they could be contained in.  So both conditions are decidable, giving the result.
\end{proof}

\begin{cor}
\label{cor atomic bountiful egs}
In of each of the following posets, \Rev{an infinite, finitely based} avoidance set $C$ is atomic if and only if its factor graph is strongly connected and for each $\sigma \in C_{[1, b-1]}$ there is $\rho \in C_{b}$ such that $\sigma \leq \rho$.
\begin{center}
(1) $\mathcal{G}$; \hspace{1.5cm} (2) $\mathcal{S}$; \hspace{1.5cm} (3) $\mathcal{D}$;\\
\hspace{-2.4cm}(4) $\mathcal{T}$; \hspace{1.5cm} (5) $\mathcal{R}_{\sigma}$.
\end{center}
\end{cor}

\begin{proof}
By Lemma \ref{lemma bountiful structures}, all of these are posets of bountiful structures, so the result follows from Corollary \ref{cor atomic bountiful}.
\end{proof}

Since Theorem \ref{thrm atomicity} was stated for posets of valid structures, it also has the following consequence beyond bountiful structures.

\begin{cor}
\label{cor nearly there}
Let $(X, \leq)$ be a poset of valid structures under the consecutive order.  If it is decidable whether a given bicycle in a factor graph of $(X, \leq)$ contains ambiguous paths, then the atomicity problem is decidable for $(X, \leq)$. 
\end{cor}

\begin{proof}
\Rev{As in the proof of Corollary \ref{cor atomic bountiful}, it is decidable whether an avoidance set is finite, and whether a finite avoidance set is atomic.  For the infinite avoidance sets,} it is easy to see that condition $(2)$ of Theorem \ref{thrm atomicity} is decidable.  For condition $(1)$, it is decidable whether the factor graph is strongly connected or a bicycle.  If it is a bicycle, by assumption, it is decidable whether it contains ambiguous paths.  Hence, it is decidable whether $(X, \leq)$ is atomic.
\end{proof}

\section{Words under the consecutive order}
\label{sec words cons}

In the remaining sections we will give an example of invalid structures -- permutations consisting of at most two ascents.  We will answer both the wqo and atomicity problems for these structures under the consecutive order.  This will involve relating these structures to certain posets of words under the consecutive order, and making use of McDevitt and Ru\v{s}kuc's solutions to the wqo and atomicity problems for words \cite{mr}.  The purpose of this section is to introduce the necessary ideas and results on words under the consecutive order, equipping us to tackle this example of an invalid type of structure.

First, we will establish how to view words as relational structures, and hence how they fit into the general framework we have established for considering consecutive orders.  A word $w$ of length $n$ over an alphabet $A=\{a_1 \dots a_m\}$ may be described as a relational structure on a set $X$ of size $n$, together with a linear order $\leq_w$, and a family of $m$ unary relations $u_1, \dots, u_{m}$.  The elements of $X$ represent the letters of $w$.  The linear order $\leq_w$ dictates the order of the letters, with $x \leq_w y$ if and only if $x$ appears to the left of $y$ in $w$.  Since we take $X \subseteq \mathbb{N}$, this linear order will be the natural one inherited from $\mathbb{N}$.  The unary relations $u_1, \dots, u_{|A|}$ indicate the letter of $A$ taken by each element of $X$: each $x \in X$ occurs in precisely one $u_i$, meaning that the letter $x = a_i$.  For example, if $A=\{a, b\}$, the word $w=abba$ can be viewed as a relational structure $(X, \leq_w, u_a, u_b)$.  Here: $X=[1, 4]$; $1 \leq_w 2 \leq_w 3 \leq_w 4$; $1, 4\in u_a$ (as the first and fourth letters are $a$); and $2, 3 \in u_b$.  

Since words have a linear order in their signature, the consecutive order is defined with respect to this linear order.

\begin{defn}
\label{defn cons subword order} 
Let $A=\{a_1 \dots a_m\}$ be an alphabet and take $w$ and $v$ to be two words over $A$, where $w = (X, \leq_w, u_1, \dots,  u_m)$ and $v=(Y, \leq_v, h_1, \dots, h_m)$.  Then $w \leq v$ under the \emph{consecutive order} if and only if there exists a contiguous mapping $f:X \rightarrow Y$ with respect to $\leq_w$ and $\leq_v$ such that $X \cong Y \Harpoon{f(X)}$.
\end{defn}

\begin{ex}
Let $A=\{a, b\}$, then $w=abba \leq baabba=v$.  To see this, we wish to embed $abba$ in the last four letters of $baabba$.  Formally, $abba = (X, \leq_w, u_a, u_b)$, as we saw earlier, and similarly $baabba = (Y, \leq_v, h_a, h_b)$, where $Y=[1, 6]$, $\leq_v$ is the natural order on $Y$, $1, 4, 5 \in h_b$ and $2, 3, 6 \in h_a$.  The required contiguous mapping $f:X \rightarrow Y$ is $f(x)=x+2$.  To see that this embeds $abba$ in $baabba$, observe that $x \in u_i$ if and only if $f(x) \in h_i$ (for $i=a, b$), indicating that $f$ preserves the value of each letter.  
\end{ex}

It can be seen that Definition \ref{defn cons subword order} mirrors the more common definition of consecutive subwords, which states that $u = u_1 \dots u_n \leq v = v_1 \dots v_m$ if and only if $u_1 \dots u_n = v_kv_{k+1} \dots v_{k+n-1}$ for some $k \in [1, m]$.  From now on, we will tend to think of consecutive subwords using this less formal definition.

We are now ready to recount the necessary results from \cite{mr}, which explore the connection between words in avoidance sets and their paths in factor graphs, building to answer the wqo and atomicity problems for $\mc{W}_A$.

\begin{prop}
\label{prop words unamb}
Let $C=\Av(B)$ be an avoidance set of $\mc{W}_A$.  Then $\Gamma_C$ contains no ambiguous paths. \qed
\end{prop}

Since words are valid by Lemma \ref{lemma valid egs}, paths in their factor graphs are always associated with at least one structure by Lemma \ref{lemma validity 2}.  Combining this with Proposition \ref{prop words unamb} gives the following consequence:

\begin{prop}
For any path $\pi$ in $\Gamma_C$, $\Sigma(\pi)$ has size one.  \qed
\end{prop}

We will abuse notation and denote the single element of $\Sigma(\pi)$ by $\Sigma(\pi)$.  

We may also combine Proposition \ref{prop words unamb} with Lemma \ref{lemma 1-1} to obtain the following result.

\begin{prop}
\label{prop words 1-1}
Take $C=\Av(B)$ to be an avoidance set of words under the consecutive order.  If $u, v \in C_{[b, \infty)}$ then $u \leq v$ if and only if $\Pi(u) \leq \Pi(v)$ under the subpath order in $\Gamma_B$. \qed
\end{prop}

From here, the wqo and atomicity problems for words reduce to those for paths under the subpath order.  The next two results follow from Propositions \ref{prop digraphs wqo}, \ref{prop digraphs atomic}, and \ref{prop words 1-1}.

\begin{prop}
\label{prop wqo words}
An avoidance set $C$ of $\mc{W}_A$ is wqo if and only if $\Gamma_C$ contains no in-out cycles.\qed
\end{prop}

\begin{prop}
\label{prop atomicity words}
An \Rev{infinite} avoidance set $C$ of $\mc{W}_A$ is atomic if and only if the following hold:
\begin{enumerate}
\item $\Gamma_C$ is strongly connected or a bicycle; and
\item For every word $w \in C_{[1, b-1]}$ there is a word $u \in C_b$ such that $w \leq u$.\qed 
\end{enumerate}
\end{prop}

Note that Proposition \ref{prop atomicity words} is the special case of Theorem \ref{thrm atomicity} for $\mc{W}_A$.  Since factor graphs of words contain no ambiguous paths by Proposition \ref{prop words unamb}, the first criterion of Theorem \ref{thrm atomicity} reduces to $\Gamma_C$ being strongly connected or a bicycle here.

The following proposition is Theorem 1.1 from \cite{mr} and follows almost immediately from Propositions \ref{prop wqo words} and \ref{prop atomicity words}.

\begin{prop}
\label{prop words}
It is decidable whether a finitely based avoidance set of words under the consecutive order is atomic or wqo. \qed
\end{prop}

\section{An invalid example: double ascents}
\label{sec double ascents}

In this section we will introduce an example of invalid \Rev{structures}: permutations with at most one inversion between consecutive entries.  These will be referred to as \emph{double ascents}.   

We will consider the poset of double ascents under the consecutive order.  Unlike in previous work, such as \cite{mr}, we will define the consecutive order for permutations in terms of their `vertical' or value linear order rather than their `horizontal' or position linear order.  For clarity, we give the formal definition of this consecutive order below.  We take this approach to make the notation consistent with the consecutive orders on other structures we have considered.  However, it turns out that the posets of permutations under the consecutive orderings defined with respect to the two linear orders are isomorphic, where the isomorphism maps each permutation to its inverse.  Therefore, our change of viewpoint does not prevent the results of \cite{mr} from fitting into our overall framework for valid structures under consecutive orders.  

Any sequence of distinct numbers $\tau$ can be reduced to a permutation by changing the smallest number into a 1, the second smallest into a 2, and so on.  As for other structures we have considered, we will say that two such sequences of numbers $\tau, \eta$ are \emph{isomorphic} if they can be reduced to the same permutation, and in this case write $\tau \cong \eta$.  We will also denote the restriction of a sequence of numbers $\tau=\tau_1\dots \tau_n$ to the points in a set $X$ by $\tau\Harpoon{X}$.  For example, $162534 \Harpoon{[2, 5]}=2534 \cong 1423$. 

As relational structures, permutations are sets with two linear orders: the `horizontal' or position linear order $\leq_p$ and the `vertical' or value linear order $\leq_v$.  For example, the permutation 15423 is a relational structure consisting of the set $[1, 5]$ with two linear orders given by $1 \leq_v 2 \leq_v 3 \leq_v 4 \leq_v 5$ and $1 \leq_p 5 \leq_p 4 \leq_p 2 \leq_p 3$.  

\begin{defn}
\Rev{An \emph{adjacent inversion}} in a permutation $\sigma=\sigma_1\dots \sigma_n$ is a pair $(\sigma_i, \sigma_{i+1})$ of consecutive entries such that $\sigma_{i+1} \leq_v \sigma_{i}$. 
\end{defn}

\begin{defn}
A permutation $\sigma=\sigma_1\dots \sigma_n$ is a \emph{double ascent} if it contains at most one adjacent inversion.  It is a \emph{single ascent} if it contains no adjacent inversions.   
\end{defn}

We will refer to increasing permutations as \emph{ascents}, so single and double ascents are permutations containing \Rev{no ascents and at most 1 ascent} as consecutive subpermutations respectively.

In previous work (e.g. \cite{mr}), the consecutive order for permutations has been defined with respect to $\leq_p$; here we will define it with respect to $\leq_v$ instead.

\begin{defn}
Let $\sigma=\sigma_1\dots \sigma_n$ and $\rho=\rho_1\dots \rho_m$ be permutations.  Then $\sigma \leq \rho$ under the \emph{consecutive order} if and only if there exists a contiguous map $f:[1,n] \rightarrow [1,m]$ with respect to $\leq_v$ such that $\sigma \cong \rho \Harpoon{im(f)}$. 
\end{defn}

\begin{ex}
$213 \cong 324 \leq 135246$ under the consecutive order \Rev{(which is defined with respect to the value linear order, rather than the more usual position order, as discussed above)}.  Here the underlying contiguous mapping with respect to $\leq_v$ is $f: [1, 3] \rightarrow [1, 6]$ defined by $f(x) = x+1$.  
\end{ex}

Since double ascents are permutations, the consecutive order for them is inherited from $\mc{P}$; let $\mc{DA}$ be the poset of double ascents with this consecutive order.

\begin{prop}
Double ascents are  not valid structures. 
\end{prop}

\begin{proof}
Consider the double ascents 13245 and 12453.  We see that 
\begin{equation*}
    13245 \Harpoon{[3, 5]} = 345 \cong 123 = 12453\Harpoon{[1, 3]}.
\end{equation*}
However, if we identify these points, we obtain the permutation 1324675, which has two adjacent inversions: those given by the subsequences 32 and 75.  Hence 1324675 is not a double ascent. 
\end{proof}

\section{Double ascents and words}
\label{sec DAs and words}

We will be able to encode double ascents as words over an alphabet of size two, enabling us to employ results from \cite{mr} for words under the consecutive order to tackle the wqo and atomicity problems for double ascents.  We will once again make use of Propositions \ref{prop wqo words} and \ref{prop atomicity words}.

To begin, we will describe the relationship between double ascents and words over $\{r, l\}$.  The idea here is that we assign entries of double ascents to the right (`$r$') or left (`$l$') according to whether they come to the right or left of the adjacent inversion.

\begin{defn}
\label{defn associated perm}
Let $w$ be a word over $\{l, r\}$.  The \emph{associated permutation} of $w$ is the double ascent $\textbf{A}(w)$ formed from the empty permutation as follows.  Reading left to right through $w$, for each letter we add a point to the permutation, in increasing value order.  If we read the letter $l$, we add the point to the rightmost position on the left of the second point \Rev{(with respect to $\leq_p$)} of the adjacent inversion, and if we read the letter $r$ we add the point to the rightmost position on the right of the first point \Rev{(again with respect to $\leq_p$)} of the adjacent inversion.
\end{defn}

In this way, the points to the left of the second point of the adjacent inversion (those assigned `left') are the points in the first ascent.  Similarly, points to the right of the first point of the adjacent inversion (those assigned `right') are the points of the second ascent.

\begin{ex}
We will construct $\textbf{A}(llrlr)$.  For clarity, we will draw a bar through the position of the adjacent inversion.  The first two letters are $l$, so the first two entries of $\textbf{A}(llrlr)$ by value are $12|$.  The next entry is given by the letter $r$, so this will be added to the right of the adjacent inversion, giving $12|3$.  Following this, the next entry is to the left of the adjacent inversion, giving $124|3$.  Finally, the last letter $r$ gives the largest entry by value to the right of the adjacent inversion, so we obtain $124|35$.  Removing the bar, $\textbf{A}(llrlr)=12435$.
\end{ex}

\begin{defn}
A double ascent $\sigma$ will be associated with the set of words 
\begin{equation*}
    \textbf{W}(\sigma)=\{w \in \{l, r\}^{+} \hspace{1mm}| \hspace{1mm} \textbf{A}(w)=\sigma\}.
\end{equation*}
\end{defn}

By inverting the process described in Definition \ref{defn associated perm}, for any double ascent $\sigma$ we can find the words in $\textbf{W}(\sigma)$.  To do this, begin with the empty word.  We will then consider the entries $\sigma$ in order of their values, smallest to largest, and for each entry add a letter to construct words in $\textbf{W}(\sigma)$.  If an entry of $\sigma$ is to the left of the second entry of the adjacent inversion, we will add the letter $l$.  On the other hand, if an entry is to the right of the first entry of the adjacent inversion, we will add the letter $r$.  It can be seen that the resulting words are in $\textbf{W}(\sigma)$ by constructing $\textbf{A}(w)$ for each word $w$ and observing that this is precisely $\sigma$.  


Since the above process can be carried out for any double ascent $\sigma$, we obtain the following corollary.

\begin{cor}
\label{cor W non-empty}
For any double ascent $\sigma$, the set $\textbf{W}(\sigma)$ is non-empty.  \qed
\end{cor}

The following proposition is immediate from the definitions.

\begin{prop}\hspace{7cm}
\begin{enumerate}
\item For any double ascent $\sigma$, if $w\in\textbf{W}(\sigma)$ then $\sigma =\textbf{A}(w)$;
\item For any word $w \in \{l, r\}^{+}$, $w \in \textbf{W}(\textbf{A}(w))$. \qed
\end{enumerate}
\end{prop}

\begin{lemma}
For a double ascent $\sigma$, $|\textbf{W}(\sigma)| > 1$ if and only if $\sigma$ is a single ascent.
\end{lemma}

\begin{proof}
$(\Leftarrow)$ Let $\sigma$ be a single ascent of length $n \geq 1$.  Then the words $l^n, r^n \in \textbf{W}(\sigma)$, so $|\textbf{W}(\sigma)| > 1$. $(\Rightarrow)$ We prove the contrapositive.  Let $\sigma$ be a double ascent, but not a single ascent.  Then the process to find the words in $\textbf{W}(\sigma)$ yields exactly one word, as each letter is determined uniquely by its position relative to the adjacent inversion in $\sigma$.
\end{proof}

From now on, in the case where $\sigma$ is not a single ascent, we will abuse notation and denote the single element of the set $\textbf{W}(\sigma)$ by $\textbf{W}(\sigma)$.

\begin{lemma}
\label{lemma single ascents}
A double ascent $\sigma$ is a single ascent of length $n$ if and only if $\textbf{W}(\sigma)=\{l^ar^c \hspace{1mm}|\hspace{1mm} a+c=n\}$.
\end{lemma}

\begin{proof}
First note that since one letter is added to words in $\textbf{W}(\sigma)$ for each entry of $\sigma$, words in $\textbf{W}(\sigma)$ have length $n$.  $(\Rightarrow)$ Consider constructing words in $\textbf{W}(\sigma)$.  Initially, we may start with either the letter $l$ or $r$.  Once we have added the letter $r$, every following letter must also be $r$, since the consecutive subword $rl$ would correspond to an adjacent inversion.  Hence, the possible words are precisely those in $\{l^ar^c \hspace{1mm} | \hspace{1mm} a+c=n\}$.  $(\Leftarrow)$ It can easily be checked that the associated permutation of any word in $\{l^ar^c \hspace{1mm} |\hspace{1mm} a+c=n\}$ is $\sigma$ (the single ascent of length $n$).
\end{proof}

\begin{lemma}
\label{lemma double ascent containment}
Let $\sigma, \rho$ be double ascents which are not single ascents. If $\sigma \leq \rho$, then $\textbf{W}(\sigma) \leq \textbf{W}(\rho)$ under the consecutive subword order.  
\end{lemma}

\begin{proof}
Since $\sigma \leq \rho$, it follows that $\sigma \cong \rho \Harpoon{[k,\hspace{0.7mm} k+|\sigma|-1]}$ for some $k \in [1, |\rho|]$.  Since neither $\sigma$ nor $\rho$ is a single ascent, $\textbf{W}(\sigma)$ and $\textbf{W}(\rho)$ are uniquely \Rev{defined}, so $\textbf{W}(\sigma)=\textbf{W}(\rho \Harpoon{[k,\hspace{0.7mm} k+|\sigma|-1]}) \leq \textbf{W}(\rho)$.
\end{proof}

\begin{lemma}
\label{lemma single ascent containment}
Suppose $\sigma, \rho$ are double ascents of which at least one is a single ascent.  If $\sigma \leq \rho$, then for any $u \in \textbf{W}(\rho)$ there exists $w \in \textbf{W}(\sigma)$ such that $w \leq u$.
\end{lemma}

\begin{proof}
If $\rho$ is a single ascent, then since $\sigma \leq \rho$, it is also true that $\sigma$ is a single ascent.  Therefore, we may assume that $\sigma$ is a single ascent.

Since $\sigma \leq \rho$, we know that $\sigma \cong \rho\Harpoon{[k,\hspace{0.7mm} k+|\sigma|-1]}$ for some $k \in \mathbb{N}$ (note that there may be more than one choice for $k$).  Take $u=u_1\dots u_n \in \textbf{W}(\rho)$.  We may take $w=u \Harpoon{[k,\hspace{0.7mm} k+|\sigma|-1]} \in \textbf{W}(\sigma)$ as required.
\end{proof}

\begin{lemma}
\label{lemma containment}
If $u, v \in \{l, r\}^{+}$ and $u \leq v$ under the consecutive subword order then $\textbf{A}(u) \leq \textbf{A}(v)$.
\end{lemma}

\begin{proof}
Suppose that $u=u_1\dots u_n$ and $v=v_1\dots v_m$.  Since $u \leq v$, $u \cong v\Harpoon{[k, \hspace{0.7mm}k+|u|-1]}$ for some $k \in \mathbb{N}$.    This means that $\textbf{A}(u)=\textbf{A}(v\Harpoon{[k, \hspace{0.7mm}k+|u|-1]}) \leq \textbf{A}(v)$ as required.
\end{proof}

\section{WQO for double ascents}
\label{sec DAs wqo}

At this point, we are ready to work towards decidability of the wqo problem for double ascents under the consecutive order.  To do this, we will exploit the relationship between double ascents and words over $\{l, r\}$, reducing the wqo problem for double ascents to certain instances of the wqo problem for $\mc{W}_{\{l, r\}}$.  

Recall that single ascents are the only double ascents which are associated with more than one word; as such, we need to treat them a little more carefully.  To begin, we will show that the sets of words associated with single ascents are wqo.

\begin{lemma}
\label{lemma Av(br)}
The words associated with single ascents are precisely those in $\Av(rl)$. 
\end{lemma}

\begin{proof}
We saw in Lemma \ref{lemma single ascents} that words associated with single ascents precisely are those of the form $l^ar^c$, where $a, c \geq 0$.  These words are exactly those avoiding $rl$ under the subword order. 
\end{proof}

\begin{lemma}
The avoidance set $\Av(rl)$ of words under the consecutive order is wqo.
\end{lemma}

\begin{proof}
$C=\Av(rl)=\{l^ar^c \hspace{1mm} | \hspace{1mm} a, c \geq 0\}$.  It can be seen that if $u=l^{a_1}r^{c_1}$ and $v=l^{a_2}r^{c_2}$ then $u \leq v$ if and only if $a_1 \leq a_2$ and $c_1 \leq c_2$.  Therefore, $(C, \leq)$ is isomorphic to $\mathbb{N} \times \mathbb{N}$ with the natural order, which is wqo.
\end{proof}

Given a finitely based avoidance set $C=\Av(B)$ of double ascents, let $\textbf{W}(B) = \{\textbf{W}(x) \hspace{1mm}| \hspace{1mm} x \in B\}$ be the set of words associated with elements of $B$.  

\begin{lemma}
\label{lemma B and W(B)}
Let $\rho$ be a double ascent.  The following are equivalent:\vspace{1mm}
\begin{thmenumerate}
    \item $\rho \in \Av(B)$;
    \item $\textbf{W}(\rho) \subseteq \Av(\textbf{W}(B))$;
    \item There exists $u \in \textbf{W}(\rho)$ such that $u \in \Av(\textbf{W}(B))$.
\end{thmenumerate}
\end{lemma}

\begin{proof}
(i) $(\Rightarrow)$ (ii) We know that $\sigma \nleq \rho$ for all $\sigma \in B$.  By the contrapositive of Lemma \ref{lemma containment}, for all $\sigma \in B$, for any $u \in \textbf{W}(\rho)$ and $t \in \textbf{W}(\sigma)$, it is the case that $t \nleq u$, meaning that $\textbf{W}(\rho) \subseteq \Av(\textbf{W}(B))$.

(ii) $(\Rightarrow)$ (iii) This implication is straightforward.

(iii) $(\Rightarrow)$ (i) Since $u \in \Av(\textbf{W}(B))$, $t \nleq u$ for all $t \in \textbf{W}(B)$.  Take $\sigma \in B$.  If neither $\sigma$ nor $\rho$ is a single ascent, it follows by the contrapositive of Lemma \ref{lemma double ascent containment} that $\sigma \nleq \rho$.  Otherwise, we can apply the contrapositive of Lemma \ref{lemma single ascent containment} to show that $\sigma \nleq \rho$.  Therefore, $\sigma \nleq \rho$ for all $\sigma \in B$, and so $\rho \in \Av(B)$.
\end{proof}

\begin{thrm}
\label{thrm words and perms}
A \Rev{finitely based} avoidance set $C=\Av(B)$ of $\mc{DA}$ is wqo if and only if the avoidance set $K=\Av(\textbf{W}(B))$ of $\mc{W}_{\{l, r\}}$ is wqo, i.e. $\Gamma_K$ contains no in-out cycles.
\end{thrm}

\begin{proof}
$(\Rightarrow)$ We will prove this direction using the contrapositive.  Suppose that $X$ is an infinite antichain in $\Av(\textbf{W}(B))$.  Now let $Y = \{x \in X\hspace{1mm} |\hspace{1mm} \textbf{A}(x) \in \Av(rl)\}$ be the subset of $X$ of words which are associated with single ascents.  Since $\Av(rl)$ is wqo, $Y$ must be finite.  We remove the elements of $Y$ from $X$.  Now let $X \backslash Y = \{w_1, w_2, \dots \}$.  By Lemma \ref{lemma double ascent containment}, $\textbf{A}(w_1), \textbf{A}(w_2), \dots $ is an infinite antichain.  By Lemma \ref{lemma B and W(B)}, $\textbf{A}(w_1), \textbf{A}(w_2), \dots $ is in $C$, and so $C$ is not wqo.

$(\Leftarrow)$ Again, we prove the contrapositive.  Suppose $\sigma_1, \sigma_2, \dots$ is an infinite antichain in $C$.  For each $i=1, 2, \dots$, let $u_i \in \textbf{W}(\sigma_i)$.  By Lemma \ref{lemma B and W(B)} and the contrapositive of Lemma \ref{lemma containment}, $u_1, u_2, \dots$ is an infinite antichain in $\Av(\textbf{W}(B))$, so $\Av(\textbf{W}(B))$ is not wqo.
\end{proof}

\begin{cor}
It is decidable whether a finitely based avoidance set of double ascents under the consecutive order is wqo. 
\end{cor}

\begin{proof}
This follows immediately from Theorem \ref{thrm words and perms} and Proposition \ref{prop words}.
\end{proof}

\begin{ex}
\label{ex d ascent non wqo}
Consider the avoidance set $C=\Av(B)$ of double ascents, where $B = \{123\}$.  Here, $\textbf{W}(B) = \{lll, llr, lrr, rrr\}$, so to determine wqo we need to look at the factor graph $\Gamma_K$ of the avoidance set $K=\Av(\textbf{W}(B))$ of $\mc{W}_{\{l, r\}}$ (shown in Figure \ref{fig first da ex}).  Since $\Gamma_K$ contains an in-out cycle, $K$ is not wqo by Proposition \ref{prop wqo words}, and so $C$ is not wqo by Theorem \ref{thrm words and perms}.
\end{ex}

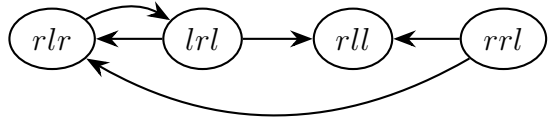
\begin{figure}
\begin{center}
\begin{tikzpicture}[> =  {Stealth [scale=1.3]}, thick]
\tikzstyle{every state}=[thick]
\node [state, shape = ellipse, minimum size = 15pt] (0) at (0, 0) {$rlr$};
\node [state, shape = ellipse, minimum size = 15pt]  (1) at (2, 0) {$lrl$};
\node [state, shape = ellipse, minimum size = 15pt]  (2) at (4, 0) {$rll$};
\node [state, shape = ellipse, minimum size = 15pt]  (3) at (6, 0) {$rrl$};

\path[->]
(0) edge [bend left] node {} (1)
(1) edge node {} (0)
(1) edge node {} (2)
(3) edge [bend left] node {} (0)
(3) edge node {} (2)
;
\end{tikzpicture}
\caption{$\Gamma_K$ from Examples \ref{ex d ascent non wqo} and \ref{ex da atomicity 2}.}
\label{fig first da ex}
\end{center}
\end{figure}

\begin{ex}
\label{ex d ascent wqo}
Now let $B = \{123, 312\}$ and consider the avoidance set $C=\Av(B)$ of double ascents.  Here, $\textbf{W}(312) = rrl$ and $\textbf{W}(123) = \{lll, llr, lrr, rrr\}$, so $\textbf{W}(B) = \{rrl, lll, llr, lrr, rrr\}$.  Now consider the avoidance set $K=\Av(\textbf{W}(B))$ of $\mc{W}_{\{r, l\}}$.  As can be seen in Figure \ref{fig second da ex}, $\Gamma_K$ contains no in-out cycles, meaning that $K$ is wqo by Proposition \ref{prop wqo words}.  Therefore, $C$ is wqo by Theorem \ref{thrm words and perms}.
\end{ex}

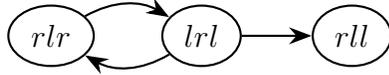
\begin{figure}
\begin{center}
\begin{tikzpicture}[> =  {Stealth [scale=1.3]}, thick]
\tikzstyle{every state}=[thick]
\node [state, shape = ellipse, minimum size = 15pt] (0) at (0, 0) {$rlr$};
\node [state, shape = ellipse, minimum size = 15pt]  (1) at (2, 0) {$lrl$};
\node [state, shape = ellipse, minimum size = 15pt]  (2) at (4, 0) {$rll$};

\path[->]
(0) edge [bend left] node {} (1)
(1) edge [bend left] node {} (0)
(1) edge node {} (2)
;
\end{tikzpicture}
\caption{$\Gamma_K$ from Example \ref{ex d ascent wqo}.}
\label{fig second da ex}
\end{center}
\end{figure}

\section{Atomicity for double ascents}
\label{sec DAs atomicity}

In this section, we will build up to proving our final theorem: decidability of the atomicity problem for double ascents under the consecutive order.  To do this, we will once again make use of the relationship between avoidance sets of double ascents and avoidance sets of words.  As for wqo, we will need to treat single ascents carefully, as they have multiple associated words rather than just one.  In this, we will start by defining the single ascents which cannot be extended to double ascents and, following this, we will explore how these single ascents translate into paths in factor graphs of words.

Throughout this section, we consider a finitely based avoidance set $C=\Av(B)$ of $\mc{DA}$.  We will take $b$ to be the maximum length of an element of $B$, and if $B=\varnothing$ we take $b=1$.  We also consider the corresponding avoidance set $K=\Av(\textbf{W}(B))$ of $\mc{W}_{\{r, l\}}$.  Note that if $B\neq \varnothing$, then $b$ is also the maximum length of an element of $\textbf{W}(B)$.

\begin{defn}
A single ascent $\sigma \in C$ is \emph{non-extendable} if $\rho\in C$ and $\sigma \leq \rho$ together imply that $\rho$ is a single ascent.
\end{defn}

\begin{defn}
A \emph{left-right bicycle} is an induced subgraph of the $b$-dimensional factor graph $\Gamma_{\Av(\textbf{W}(B))}$ all of whose vertices are of the form $l^ar^c$ such that $a+c=b$.  
\end{defn}

We will see in a moment that, as their name suggests, left-right bicycles are always bicycles.

\begin{ex}
\label{ex rb bicycle}
Consider the avoidance set $K=\Av(rl)$ of $\mc{W}_{\{r, l\}}$.  The factor graph $\Gamma_K$ is shown in Figure \ref{fig rb bic ex}, and is a left-right bicycle.
\end{ex}

\begin{lemma}
    All left-right bicycles are bicycles.
\end{lemma}

\begin{proof}
    Consider an avoidance set $C=\Av(B)$ of double ascents and its associated avoidance set $K$.  Let $X$ be a left-right bicycle in $\Gamma_{K}$.  Since $X$ contains at least one vertex, which must be of the form $l^ar^c$ such that $a+c=b$, $C$ contains the single ascent $\sigma$ of length $b$ by Lemma \ref{lemma B and W(B)}.  Indeed, Lemma \ref{lemma B and W(B)} tells us that all elements of $\textbf{W}(\sigma)$ are vertices of $\Gamma_{K}$.  These are precisely the words of the form $l^ar^c$ such that $a+c=b$.  Therefore, $X$ contains all of these vertices.  It follows that $X$ is the bicycle whose initial cycle is a loop on $l^b$, whose terminal cycle is a loop on $r^b$ and whose connecting path introduces $r$'s one at a time.  
\end{proof}

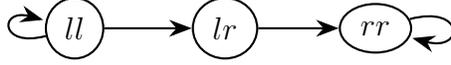
\begin{figure}
\begin{center}
\begin{tikzpicture}[> =  {Stealth [scale=1.3]}, thick]
\tikzstyle{every state}=[thick]
\node [state, shape = ellipse, minimum size = 15pt]  (0) at (0, 0) {$ll$};
\node [state, shape = ellipse, minimum size = 15pt]  (1) at (2, 0) {$lr$};
\node [state, shape = ellipse, minimum size = 15pt]  (2) at (4, 0) {$rr$};

\path[->]
(0) edge [loop left] node {} (0)
(0) edge node {} (1)
(2) edge [loop right] node {} (2)
(1) edge node {} (2)
;
\end{tikzpicture}
\caption{$\Gamma_K$ from Examples \ref{ex rb bicycle} and \ref{ex da atomicity 1}.}
\label{fig rb bic ex}
\end{center}
\end{figure}
\clearpage

\begin{lemma}
\label{lemma non-extendables}
If an \Rev{infinite} avoidance set $C=\Av(B)$ of double ascents contains non-extendable single ascents, the following are equivalent:
\begin{enumerate}
\item $C$ is atomic;
\item All elements of $C$ are single ascents (i.e. $C \subseteq \Av(21)$);
\item $(i)$ $\Gamma_{K}$ is a left-right bicycle and $(ii)$ for any $\sigma \in C_{[1, \hspace{0.7mm}b-1]}$ there exists $\rho \in C_b$ such that $\sigma \leq \rho$.
\end{enumerate}
\end{lemma}

\begin{proof}
$(1) \Rightarrow (2)$ Since non-extendable single ascents can only be contained in single ascents, the JEP immediately implies that all elements of $C$ are single ascents.

$(2) \Rightarrow (1)$ Since single ascents form a chain, $C$ must be atomic.

$(3) \Rightarrow (2)$ Since $\Gamma_{K}$ is a left-right bicycle, the only words of length $\geq b$ in $\Av(\textbf{W}(B))$ are of the form $l^ar^c$, where $a+c \geq b$.  All of these words are associated with single ascents by Lemma \ref{lemma single ascents}, so all elements of $C_{[b, \infty)}$ are single ascents by Lemma \ref{lemma B and W(B)}.  It follows from $(3)(ii)$ that in fact all elements of $C$ are single ascents.

$(2) \Rightarrow (3)$ $B$ must contain 21, meaning that $rl \in \textbf{W}(B)$.  This means that vertices of $\Gamma_{K}$ avoid $rl$, so are of the form $l^ar^c$ where $a+c = b$.  Therefore, $\Gamma_{\textbf{W}(B)}$ is a left-right bicycle.  The second condition follows since $C$ is atomic, as we know $(1)$ is equivalent to $(2)$.
\end{proof}

\begin{prop}
\label{thrm no non-extendables}
If an \Rev{infinite} avoidance set $C=\Av(B)$ of double ascents contains no non-extendable single ascents, then $C$ is atomic if and only if the following hold:\vspace{1mm}
\begin{thmenumerate}
\item $\Gamma_{K}$ is strongly connected or a bicycle;
\item For any double ascent $\sigma \in C_{[1, b-1]}$, there is a double ascent $\rho \in C_b$ such that $\sigma \leq \rho$.
\end{thmenumerate}
\end{prop}

\begin{proof}
$(\Leftarrow)$ Take $\sigma^-, \rho^- \in C$ and extend them via (ii) to $\sigma, \rho \in C_{[b, \infty)}$ respectively, which we can take not to be single ascents since $C$ contains no non-extendable single ascents.  

\Rev{It} follows from (i) and Proposition \ref{prop digraphs atomic} that \Rev{the set of paths of} $\Gamma_{K}$ is atomic. Therefore there exists a path $\pi$ in $\Gamma_{K}$ such that $\Pi(\textbf{W}(\sigma)), \Pi(\textbf{W}(\rho)) \leq \pi$.  By Proposition \ref{prop words 1-1}, $\textbf{W}(\sigma), \textbf{W}(\rho) \leq \Sigma(\pi)$ in $K$.  We observe that $\sigma, \rho \leq \textbf{A}(\Sigma(\pi))$ in $C$ by Lemma \ref{lemma containment}.  Therefore $C$ satisfies the JEP and so is atomic.

$(\Rightarrow)$ First we will prove that (ii) holds.  Suppose to the contrary that there exists $\sigma \in C_{[1, b-1]}$ which is not contained in any element of $C_b$.  Pick an element $\rho \in C_b$.  Since $C$ is atomic, there exists $\tau \in C$ such that $\sigma, \rho \leq \tau$.  As $\rho \leq \tau$, we know that $|\tau| \geq b$. Further, since $\sigma \leq \tau$, there is a consecutive subpermutation $\tau_b$ of $\tau$ of length $b$ such that $\sigma \leq \tau_b$, a contradiction.  We conclude that (ii) holds.

\Rev{Now we will prove that (i) holds.  Again, suppose to the contrary that $\Gamma_{K}$ is neither strongly connected nor a bicycle.  By Proposition \ref{prop digraphs atomic}, the set of paths of $\Gamma_{K}$ is not atomic, so we may take paths $\eta, \pi$ in $\Gamma_{K}$ which do not join.  It follows that $\Sigma(\eta), \Sigma(\pi)$ do not join in $K$ by Proposition \ref{prop words 1-1}.  Therefore, $\Sigma(\eta), \Sigma(\pi)$ cannot both be associated with single ascents, because their paths would be in the same left-right bicycle, and so would join.  If needed, extend $\Sigma(\eta)$ to be the word of a double ascent in $C$ which is not a single ascent, call this $w$.  Note that $\Sigma(\eta)\leq w$.  Since $C$ is assumed to be atomic, there exists $\theta \in C$ such that $\textbf{A}(w), \textbf{A}(\Sigma(\pi)) \leq \theta$.  Therefore, by Lemma \ref{lemma double ascent containment}, $w, \Sigma(\pi) \leq \textbf{W}(\theta)$.  Hence $\Sigma(\eta), \Sigma(\pi) \leq \textbf{W}(\theta)$, yielding a contradiction as $\Sigma(\eta)$ and $\Sigma(\pi)$ do not join.  We conclude that $\Gamma_{K}$ must be strongly connected or a bicycle.}
\end{proof}

\begin{defn}
A left-right bicycle in a factor graph $\Gamma_{K}$ is \emph{isolated} if it is a connected component of $\Gamma_{K}$.
\end{defn}

\begin{lemma}
\label{lemma non-extendable decidable}
$C_{[b, \infty)}$ contains non-extendable single ascents if and only if $\Gamma_{K}$ contains an isolated left-right bicycle. 
\end{lemma}

\begin{proof}
$(\Rightarrow)$ Let $\sigma \in C_{[b, \infty)}$ be a non-extendable single ascent.  By Lemma \ref{lemma single ascents} any word $w \in \textbf{W}(\sigma)$ is of the form $w=l^ar^c$ where $a+c=b$, and this cannot be extended to a word which is not of this form.  Hence in $\Gamma_{K}$, the path $\Pi(\textbf{W}(\sigma))$ must be wholly contained in a left-right bicycle $X$.  Moreover, since any possible extension of $\textbf{W}(\sigma)$ is of the form $l^ar^c$, any extension of $\Pi(\textbf{W}(\sigma))$ must also be wholly contained in $X$.  Therefore, in $\Gamma_{K}$, $X$ cannot be connected to any vertices outside of itself.  It follows that $X$ is isolated.

$(\Leftarrow)$ Let $X$ be an isolated left-right bicycle in $\Gamma_{K}$.  Then any word associated with a path in $X$ must be of the form $w=l^ar^c$ where $a+c=b$.  By Lemma \ref{lemma single ascents}, $w$ must be associated with a single ascent.  Since no path $\pi$ in $X$ can be extended to a path that leaves $X$, no word $\Sigma(\pi)$ can be extended to a word not of the form $l^ar^c$ where $a+c=b$.  Therefore, single ascents associated with $\Sigma(\pi)$ cannot be extended to double ascents.
\end{proof}

\begin{thrm}
\label{thrm da atomicity}
Let $C=\Av(B)$ be an \Rev{infinite,} finitely based avoidance set of double ascents under the consecutive order and $K=\Av(\textbf{W}(B))$.  Then $C$ is atomic if and only if the following hold:
\begin{enumerate}
\item $\Gamma_{K}$ is strongly connected or a bicycle; and 
\item For any double ascent $\sigma \in C_{[1, b-1]}$ there is a double ascent $\rho \in C_b$ such that $\sigma \leq \rho$. 
\end{enumerate}
\end{thrm}

\begin{proof}
If $C$ contains no non-extendable single ascents, the result is given by Proposition \ref{thrm no non-extendables}.  Now consider the case where $C$ contains non-extendable single ascents.  

$(\Rightarrow)$ By Theorem \ref{lemma non-extendables}, $\Gamma_{K}$ is a left-right bicycle, so it is certainly a bicycle and condition $(1)$ holds.  We will prove that $(2)$ holds by contradiction.  Suppose that $\sigma \in C_{[1, b-1]}$ such that $\sigma$ is not contained in any element of $C_b$.  Let $\nu \in C_b$.  Since $C$ is atomic, there exists $\theta \in C_{[b, \infty)}$ such that $\sigma, \nu \leq \theta$.  Since $|\theta| \geq b$, we see that $\sigma$ must be contained in a factor of $\theta$ of length $b$, a contradiction.

$(\Leftarrow)$ For the reverse direction, first note that if $C_{[1, b-1]}$ contains non-extendable single ascents, so does $C_{[b, \infty)}$ by condition $(2)$.  By Lemma \ref{lemma non-extendable decidable}, $\Gamma_{K}$ contains an isolated left-right bicycle.  Since by assumption $\Gamma_{K}$ is either strongly connected or a bicycle, it must be an isolated left-right bicycle.  Since (2) gives condition (3)(ii) of Theorem \ref{lemma non-extendables}, the result follows from Theorem \ref{lemma non-extendables}.  
\end{proof}

\begin{thrm}
It is decidable whether a finitely based avoidance set $C=\Av(B)$ of double ascents under the consecutive order is atomic.
\end{thrm}

\begin{proof}
\Rev{We can determine whether $C$ is finite by checking whether the factor graph $\Gamma_K$ of $K=\Av(\textbf{W}(B))$ contains cycles.  If $C$ is finite, atomicity can be determined algorithmically.  If $C$ is infinite,} the conditions of Theorem \ref{thrm da atomicity} can be tested algorithmically, so the result follows.
\end{proof}

\begin{ex}
\label{ex da atomicity 1}
Consider the avoidance set $C=\Av(21)$ of double ascents.  Here $\textbf{W}(21)=rl$, so we need to look at the factor graph of the avoidance set $K=\Av(rl)$, which is precisely the one we saw in Figure \ref{fig rb bic ex}.  Since $\Gamma_K$ is a (left-right) bicycle, and it is easy to check condition (2) of Theorem \ref{thrm da atomicity}, $C$ is atomic by Theorem \ref{thrm da atomicity}.
\end{ex}

\begin{ex}
\label{ex da atomicity 2}
Let us return to the avoidance set $C=\Av(B)$ of double ascents, where $B = \{123\}$ from Example \ref{ex d ascent non wqo}.  Here, $K=\Av(lll, llr, lrr, rrr)$ and $\Gamma_K$ is given in Figure \ref{fig first da ex}.  Since $\Gamma_K$ is neither strongly connected nor a bicycle, $C$ is not atomic by Theorem \ref{thrm da atomicity}.
\end{ex}

\begin{ex}
\label{ex da atomicity 3}
Let $X=\{\Rev{2314}, 1243, 1342, 1324, 2413\}$ and let $S$ be the set of double ascents on 4 points.  Let $B=S\backslash X$ and consider the avoidance set $C=\Av(B)$.  It can be seen that $\textbf{W}(B)$ contains all of the words of length four over $\{l, r\}$ except for those associated with elements of $X$: $rllr, llrl, lrll, lrlr, rlrl$.  These will be the vertices of $\Gamma_K$ (shown in Figure \ref{fig da atomicity 3}).  Here, condition (2) of Theorem \ref{thrm da atomicity} does not hold, since there is no element of $C_4$ which contains $\Rev{312}$.  Therefore, $C$ is not atomic by Theorem \ref{thrm da atomicity}.
\end{ex}

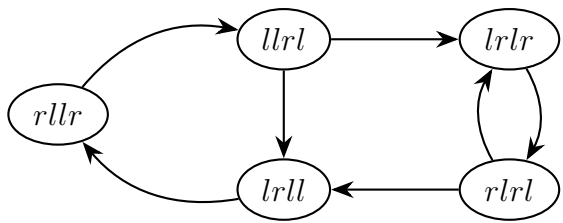
\begin{figure}
\begin{center}
\begin{tikzpicture}[> =  {Stealth [scale=1.3]}, thick]
\tikzstyle{every state}=[thick]
\node [state, shape = ellipse, minimum size = 15pt] (0) at (0, 1) {$rllr$};
\node [state, shape = ellipse, minimum size = 15pt]  (1) at (3, 2) {$llrl$};
\node [state, shape = ellipse, minimum size = 15pt]  (2) at (3, 0) {$lrll$};
\node [state, shape = ellipse, minimum size = 15pt]  (3) at (6, 2) {$lrlr$};
\node [state, shape = ellipse, minimum size = 15pt]  (4) at (6, 0) {$rlrl$};

\path[->]
(0) edge [bend left] node {} (1)
(1) edge node {} (2)
(2) edge [bend left] node {} (0)
(1) edge node {} (3)
(3) edge [bend left] node {} (4)
(4) edge [bend left] node {} (3)
(4) edge node {} (2)
;
\end{tikzpicture}
\caption{$\Gamma_K$ from Example \ref{ex da atomicity 3}.}
\label{fig da atomicity 3}
\end{center}
\end{figure}

\section{Concluding remarks and open problems}
\label{sec conclusion}

We begin with a brief discussion of our current knowledge of wqo for posets of combinatorial structures under consecutive orders.  In this paper, we have established decidability of the wqo problem for posets of bountiful structures under consecutive orders, and we have seen that these are posets for which all paths in factor graphs are ambiguous.  Key to this was the result that ambiguous cycles give rise to infinite antichains.

On the other hand, posets of valid structures for which all paths in factor graphs are unambiguous also have decidable wqo problems.  In these cases, structures and paths are in bijective correspondence, and this also preserves the orderings.  Hence, we can simply use Proposition \ref{prop digraphs wqo} to determine whether there are infinite antichains of paths, which must be in bijective correspondence with infinite antichains of structures.  In this, in-out cycles yield non-wqo.  Examples of structures in this category are words (shown in \cite{mr}) and linear orders. 

Now we may ask about wqo for \emph{intermediate structures}: valid structures whose factor graphs may contain a mixture of ambiguous and unambiguous paths.  Examples of these include permutations (studied in \cite{mr}) and equivalence relations (see \cite{ir}).  Here, both ambiguous cycles and in-out cycles still give rise to non-wqo.  In the case of equivalence relations, these are the only conditions on the factor graph for wqo, whereas for permutations there is an additional condition requiring that there are no \emph{splittable pairs}.  Our first questions invite investigation into wqo for intermediate structures and how this may relate to the cases discussed already.  It is easy to see that posets are intermediate structures; however, unlike for permutations and equivalence relations, the wqo and atomicity problems are open for posets, leading to our first question.

\begin{que}
    Are the wqo and atomicity problems decidable for posets under the consecutive order?  
\end{que}

\begin{que}
What are examples of other intermediate structures?  Can we answer the wqo problem for them using factor graphs, and, if so, what are the conditions required for wqo?  
\end{que}

\begin{que}
Is there an overarching framework encompassing the wqo results for bountiful structures, structures with no ambiguous paths, and intermediate structures?
\end{que}

Turning to the property of atomicity, we have proved decidability of the atomicity problem for bountiful structures.  Meanwhile, Corollary \ref{cor nearly there} reduces the atomicity problem for any poset of valid structures under the consecutive order to the following question.  If answered in the affirmative, this would yield decidability of atomicity for all posets of valid structures under the consecutive order.

\begin{que}
Is it decidable whether a given bicycle in a factor graph of a poset of valid structures contains ambiguous paths?
\end{que}

A further avenue for future research involves posets of invalid structures under consecutive orders.  In this paper we have shown the wqo and atomicity problems to be decidable for two types of invalid structures: overlap free words and double ascents.  For each of these, we exploited connections with the poset of words under the consecutive order.  We also noted in Example \ref{ex factor graphs don't work} that forests are invalid structures; the solution to the wqo problem for forests is given in the first author's thesis \cite{VIThesis}, and proving this involves an extension of factor graphs. 

\begin{que}
Is the atomicity problem decidable for forests under the consecutive order?
\end{que}

\begin{que}
What are other examples of invalid structures, and can we solve the wqo and atomicity problems for them?  Are there connections between these results and results for valid structures under consecutive orders?
\end{que}

\Rev{\section*{Acknowledgements}
We would like to thank the two anonymous reviewers for their careful reading of this paper and their helpful suggestions to improve it.}

\nocite{*}
\bibliographystyle{abbrvnat}
\bibliography{Bountiful}
\label{sec:biblio}

\end{document}